\documentclass[ejs]{imsart}

\usepackage{tikz}
\usepackage{graphicx}
\usepackage{amsfonts,amssymb,amsmath,amsthm}
\usepackage[numbers,square]{natbib}
\usepackage{hyperref}
\hypersetup{backref, colorlinks=true, citecolor=blue, linkcolor=blue}

\usepackage{pdfsync}
\usepackage{hypernat}

\usepackage[ruled,lined]{algorithm2e}
\arxiv{1109.5998}

\startlocaldefs
\definecolor{greenmain}{HTML}{00AF64}
\definecolor{bluemain}{HTML}{0B61A4}
\definecolor{orangemain}{HTML}{FF9200}
\definecolor{redmain}{HTML}{FF4900}

\usepackage{aliascnt}
\newtheorem{theorem}{Theorem}

\newaliascnt{lemma}{theorem}
\newtheorem{lemma}[lemma]{Lemma}
\aliascntresetthe{lemma}

\newtheorem{proposition}[theorem]{Proposition}
\newaliascnt{cor}{theorem}

\aliascntresetthe{cor}

\newaliascnt{def}{theorem}
\newtheorem{definition}[def]{Definition}
\aliascntresetthe{def}

\def\indep{\perp\!\!\!\perp}

\newcommand{\F}{\mathcal{F}}
\newcommand{\E}{\mathbb{E}}
\renewcommand{\P}{\mathbb{P}}
\renewcommand{\tilde}{\widetilde}
\newcommand{\R}{\mathbb{R}}

\newcommand{\V}{\mathbb{V}}

\newcommand{\TV}[1]{\lVert #1 \rVert_{TV}}

\newcommand{\Pij}[2]{\P_{#1:#2}}

\endlocaldefs

\begin{document}

\begin{frontmatter}

\title{Estimating beta-mixing coefficients via histograms}

\runtitle{Estimating beta-mixing coefficients}

\begin{aug}
 \author{\fnms{Daniel J.} \snm{McDonald}\corref{}\thanksref{c} \ead[label=e1]{dajmcdon@indiana.edu}}
 \author{\fnms{Cosma Rohilla} \snm{Shalizi}\thanksref{a,b,e2}\ead[label=e2,mark]{cshalizi@cmu.edu}}
 \and
 \author{\fnms{Mark}
   \snm{Schervish}\thanksref{a,e3}\ead[label=e3,mark]{mark@cmu.edu}}
 \address[c]{Department of Statistics\\ Indiana University\\
   Bloomington, IN 47401\\\printead{e1}}
 \address[a]{Department of Statistics\\ Carnegie Mellon University\\
   Pittsburgh, PA 15213\\\printead{e2,e3}}
 \address[b]{Santa Fe Institute\\ 1399 Hyde Park Road\\
   Santa Fe, New Mexico 87501}

\runauthor{McDonald, Shalizi, and Schervish}

\affiliation{Department of Statistics, Carnegie Mellon University}

\end{aug}

\begin{abstract}
  The literature on statistical learning for time series often assumes
  asymptotic independence or ``mixing'' of the data-generating process. These
  mixing assumptions are never tested, nor are there methods for estimating
  mixing coefficients from data. Additionally, for many common classes of
  processes (Markov processes, ARMA processes, etc.) general functional forms
  for various mixing rates are known, but not specific coefficients.  We
  present the first estimator for beta-mixing coefficients based on a single
  stationary sample path and show that it is risk consistent. Since mixing
  rates depend on infinite-dimensional dependence, we use a Markov
  approximation based on only a finite memory length $d$. We present
  convergence rates for the Markov approximation and show that as
  $d\rightarrow\infty$, the Markov approximation converges to the true mixing
  coefficient. Our estimator is constructed using $d$-dimensional histogram
  density estimates. Allowing asymptotics in the bandwidth as well as the
  dimension, we prove $L^1$ concentration for the histogram as an intermediate
  step. Simulations wherein
  the mixing rates are calculable and a real-data example demonstrate
  our methodology. 
\end{abstract}





\end{frontmatter}

\section{Introduction}
\label{sec:introduction}

The ordinary theory of statistical inference is overwhelmingly concerned with
independent observations, but the exact work done by assuming independence is
often mis-understood.  It is not, despite a common impression, to guarantee
that large samples are representative of the underlying population, ensemble,
or stochastic source.  If that were all that were needed, one could use the
ergodic theorem for dependent sources equally well.  Rather, assuming
independence lets statistical theorists say something about the {\em rate} at
which growing samples approximate the true distribution.  Under statistical
independence, every observation is completely unpredictable from every other,
and hence provides a completely new piece of information about the source.
Consequently, the most common measures of information --- including the
Kullback-Leibler divergence between probability measures, the Fisher
information about parameters, and the joint Shannon entropy of random variables
--- are all strictly proportional to the number of observations for IID
sources.  Under dependence, later events are more or less predictable from
earlier ones, hence they do {\em not} provide completely new observations, and
information accumulates more slowly.  Assuming ergodicity alone, the
convergence of samples on the source can be arbitrarily slow, and statistical
theory is crippled.  Without more stringent assumptions than ergodicity, one is
always effectively in an $n=1$ situation no matter how many observations one
has.

To go beyond independence, statistical theory needs assumptions on the
data-generating processes which control the rate at which information
accumulates.  For time series analysis, the most natural replacement for
independence is requiring the asymptotic independence of events far apart in
time, or {\bf mixing}.  Mixing quantifies the decay in dependence as the future
moves farther from the past.  There are many definitions of mixing of varying
strength with matching dependence coefficients \citep[see][for
reviews]{Doukhan1994,DedeckerDoukhan2007,Bradley2005}, but many of the results
in the statistical literature focus on $\beta$-mixing or absolute regularity.
Roughly speaking (see \autoref{defn:beta-mix} below for a precise statement),
the $\beta$-mixing coefficient at lag $a$ is the total variation distance
between the actual joint distribution of events separated by $a$ time steps and
the product of their marginal distributions, i.e., the $L^1$ distance from
independence.

Much of the theoretical groundwork for the analysis of mixing processes was
laid years ago \citep{Withers1981, Bradley1983, Eberlein1984, PhamTran1985,
  AthreyaPantula1986a, Tran1989, Yu1993, Yu1994}, but it remains an active
topic in probability, statistics and machine learning.  Among the many works on
this topic, we may mention the study of non-parametric inference under mixing
conditions by \citet{Bosq1998}, consistent time series forecasting by support
vector machines \citep{SteinwartAnghel2009}, probably approximately correct
learning algorithms with mixing inputs
\citep{Vidyasagar1997,KarandikarVidyasagar2009} and stability-based
generalization error bounds \citep{MohriRostamizadeh2010}.  To actually {\em
  use} such results, however, requires knowing the $\beta$-mixing coefficients,
$\beta(a)$.

Many common time series models are known to be $\beta$-mixing, and the rates of
decay are known up to constant factors given the true parameters of the
process. Among the processes for which such results exist are ARMA models
\citep{Mokkadem1988}, GARCH models \citep{CarrascoChen2002}, and certain Markov
processes --- see \citet{Doukhan1994} for an overview of such results.
(\citet{FryzlewiczSubba-Rao2011} derive upper bounds for the $\alpha$- and
$\beta$-mixing rates of non-stationary ARCH processes.)  With few exceptions,
however, these results do not give the actual mapping from parameters to mixing
coefficients.  For example, it is known that the mixing coefficients of the
ARMA process at time lag $a$ are $O(\rho^{a})$ for some $0<\rho < 1$. While
knowledge of the ARMA parameters determine the joint distribution, no one has
yet figured out how to map from these parameters to the constants, or even to
$\rho$. To our knowledge, only \citet{Nobel2006} approaches a solution to the
problem of estimating mixing rates by giving a method to distinguish between
different polynomial mixing rate regimes through hypothesis testing.  Thus the
theoretical results which presume that the mixing coefficients are known cannot
actually be applied to assist in the analysis of data, or even of parametric
models.

These issues also arise in interpreting the output of Markov chain Monte Carlo
(MCMC) algorithms.  Even when the Markov chain is in equilibrium, sampling from
the desired invariant distribution, the samples are dependent.  How much
dependence persists across samples is a very important issue for users wishing
to control Monte Carlo error, or planning how long a run they need.  In some
rare cases, theoretical results show that certain MCMC algorithms are rapidly
mixing, meaning again roughly that $\beta(a) = O(\rho^{a})$.  Such results
generally do not give $\rho$, let alone $\beta(a)$, which is what users
would need.

We present the first method for estimating the $\beta$-mixing coefficients for
stationary time series data given a single sample path. Our methodology can be
applied to real data assumed to be generated from some unknown $\beta$-mixing
process.  Additionally, it can be used to examine known mixing processes,
thereby determining exact mixing rates via simulation.  (This includes, but is
not limited to, MCMC algorithms.)  \autoref{sec:estimation} defines the
$\beta$-mixing coefficient, our estimator of it, and states our main results
on convergence rates and consistency for the estimator.
\autoref{sec:conv-hist} gives an intermediate result on the $L^1$ convergence
of the histogram estimator with $\beta$-mixing inputs which is asymptotic in
the dimension of the target distribution in addition to the bandwidth. Some of
our results and techniques here are of independent interest for
high-dimensional density estimation.  \autoref{sec:proof} proves the main
results from \autoref{sec:estimation}.  \autoref{sec:examples} demonstrates the
performance of our estimator in three simulated examples, providing good
recovery of known rates in simple settings as well as providing insight into
unknown mixing regimes, and also examines a dataset containing
recession indicators for developed economies.
\autoref{sec:discussion} concludes and lays out some 
avenues for future research.

\section{Estimator and consistency results}
\label{sec:estimation}


In this section, we present one of many equivalent definitions of absolute
regularity and state our main results, deferring proof to \S \ref{sec:proof}.

To fix notation, let $\mathbf{X}=\{X_t\}_{t=-\infty}^\infty$ be a sequence of
random variables where each $X_t$ is a measurable function from a probability
space $(\Omega,\F,\P)$ into $\mathbb{R}^q$ with the Borel
$\sigma$-field $\mathcal{B}$.  A
block of this random sequence will be given by $\mathbf{X}_{i:j} \equiv
\{X_t\}_{t=i}^j$ where $i$ and $j$ are integers, and either may be infinite.
We use similar notation for the sigma fields generated by these
blocks.  In particular, $\sigma_{i:j}$ will denote the sigma 
field generated by $\mathbf{X}_{i:j}$, and the joint distribution of
$\mathbf{X}_{i:j}$ will be denoted $\Pij{i}{j}$. 
We denote products of marginal distributions as, e.g., $\Pij{i}{j}
\otimes \Pij{k}{l}$. 

\subsection{Definitions}
\label{sec:definitions}

In this paper, we will consider only the case of stationary data. 
\begin{definition}[Stationarity]\label{def:stationary}
  A sequence of random variables $\mathbf{X}$ is {\em stationary} when all its
  finite-dimensional distributions are invariant over time: for all integers
  $t,\ t'$ and all
  non-negative integers $i$, the distribution of $\mathbf{X}_{t:t+i}$
  is the same as that of $\mathbf{X}_{t':t'+i}$.
\end{definition}

There are many equivalent definitions of $\beta$-mixing (see for
instance~\citep{Doukhan1994}, or~\citep{Bradley2005} as well
as~\citep{Meir2000} or~\citep{Yu1994}), however the most intuitive is that
given in~\citet{Doukhan1994}.
\begin{definition}[$\beta$-mixing]
  \label{defn:beta-mix}
  For each $a \in \mathbb{N}$, the {\em coefficient
    of absolute regularity}, or {\em $\beta$-mixing coefficient}, $\beta(a)$,
  is
  \begin{equation}
    \label{eq:13}
    \beta(a) := \TV{\Pij{-\infty}{0} \otimes \Pij{a}{\infty} -
      \P_{-\infty:0, a:\infty}} 
  \end{equation}
  where $\| \cdot \|_{TV}$ is the total variation norm, and
  $\P_{-\infty:0,\ a:\infty}$ is the joint distribution of the blocks
  $(\mathbf{X}_{-\infty:0},\ \mathbf{X}_{a:\infty})$. A
  stochastic process is said to be {\em absolutely regular}, or {\em
    $\beta$-mixing}, if $\beta(a) \rightarrow 0$ as
  $a\rightarrow\infty$.
\end{definition}

Loosely speaking, \autoref{defn:beta-mix} says that the coefficient $\beta(a)$
measures the total variation distance between the joint distribution of random
variables separated by $a$ time units, $\P_{-\infty:0,a:\infty}$
and the distribution under which random 
variables separated by $a$ time units are independent,
$\Pij{-\infty}{0} \otimes \Pij{a}{\infty}$.  We note that in the most
general setting in the literature, $\beta(a) = \sup_t \TV{\Pij{-\infty}{t} \otimes \Pij{t+a}{\infty}
  - \P_{-\infty:t,t+a:\infty}}$, however, this additional generality is unnecessary for
stationary random processes $\mathbf{X}$, which is the only case we consider
here. 

As stationarity implies that distributions of blocks of random
variables are the same, and we will frequently require notation for
these distributions, we will employ the following simplifications: the
distribution of a $d$-block will be notated $\P_{[d]} =
\Pij{i}{i+d}=\Pij{j}{j+d}$ and the joint distribution of two blocks of
length $d$ separated by $a$ timepoints,
$(\mathbf{X}_{i:(i+d-1)},\mathbf{X}_{(i+d+a-1):(i+2d+a-1)})$, will be
given by $\P_{[d],a}$. In particular, $\P_{-\infty:0, a:\infty}$ will
be written as $\P_{[\infty],a}$ and similarly for the associated
sigma-fields when necessary.

\subsection{Constructing the estimator}
\label{sec:estimator}

Our result emerges in two stages. First, we recognize that the distribution of
a finite sample depends only on finite-dimensional distributions.  This leads
to an estimator of a finite-dimensional version of $\beta(a)$.  Next, we let
the finite-dimension increase to infinity with the size of the observed sample.

For positive integers $d$, and $a$, define
\begin{equation}
  \label{eq:2a}
  \beta^d(a)=\TV{\P_{[d]} \otimes \P_{[d]}-
    \P_{[d],a}}.
\end{equation}
 Also, let $\widehat{f}^d$ be the
histogram estimator of the joint density of $d$ consecutive
observations, that is
\begin{equation*}
  \widehat{f}^d(\mathbf{x}) = \frac{1}{(n-d+1)h_n^d} \sum_{i=1}^{n-d+1}{ I(\mathbf{X}_{i:i+d-1} \in B(\mathbf{x})) }
\end{equation*}
where $B(\mathbf{x})$ is the bin containing $\mathbf{x}$ and $I(\cdot)$ is the
indicator function. Similarly, let $\widehat{f}_a^{2d}$ be the $2d$-dimensional
histogram estimator of the joint density of two sets of $d$ consecutive
observations separated by $a$ time points, i.e
\begin{align*}
  \lefteqn{\widehat{f}_a^{2d}(\mathbf{x})}\\ &=
                                               \frac{1}{(n-2d-a+1)h_n^{2d}}
                                               \sum_{i=1}^{n-2d-a+1} 
  I((\mathbf{X}_{i:(i+d-1)},\mathbf{X}_{(i+d+a-1):(i+2d+a-1)}) \in
                                               B(\mathbf{x})). 
\end{align*}
Note that as we have assumed $\mathbf{X}_i \in \R^q$, in the above
definitions, $\mathbf{x}\in \R^{dq}$ and $\mathbf{x}\in\R^{2dq}$
respectively with $h=h^q$. As we assume $q$ fixed throughout, we
suppress this dependence. We discuss this issue further in a remark
following our main result in the next subsection.

We construct an estimator of $\beta^d(a)$ based on these two
histograms.\footnote{While it is clearly possible to replace histograms with
  other choices of density estimators (most notably kernel density estimators),
  histograms in this case are more convenient theoretically and
  computationally. See \S \ref{sec:discussion} for more details.} Define
\begin{equation}
  \label{eq:estimator}
  \widehat{\beta}^{d}(a) = \frac{1}{2}\int\left|\widehat{f}_a^{2d} -
    \widehat{f}^d\otimes\widehat{f}^d\right|
\end{equation}
We will show that, by having $d=d_n$ grow (slowly) with $n$, this estimator will
converge to $\beta(a)$. This can be seen most clearly by bounding the risk of
the estimator with its estimation and approximation errors:
\[
|\widehat{\beta}^{d}(a) - \beta(a)| \leq |\widehat{\beta}^{d}(a) -
\beta^{d}(a)| + |\beta^{d}(a) -  \beta(a)|.
\]
The first term is the error of estimating $\beta^d(a)$ from a random
sample. The second term is the non-stochastic error induced by
approximating the infinite dimensional coefficient, $\beta(a)$, by
its $d$-dimensional counterpart, $\beta^d(a)$.

\subsection{Assumptions and main results}
\label{sec:assumpt-main-results}

The results of this paper require two main assumptions. The first is that the
process $\mathbf{X}_1^n$ is generated by a stationary, $\beta$-mixing
distribution with density $f$. Second, we must place some conditions on the
density $f$ to ensure that the histogram estimators $\widehat{f}^d$ and
$\widehat{f}^{2d}$ will actually converge to the densities $f^d$ and $f^{2d}$.
Specifically, we will assume continuity and regularity conditions as in
\citep{FreedmanDiaconis1981b}\footnote{We discuss modifications for discrete
  distributions below, p.\ \pageref{about-discrete-processes}.}:
\begin{enumerate}
\item $f \in L^2$ and $f$ is absolutely continuous on its support, with
  a.e.\ partial derivatives $f_i=\frac{\partial}{\partial y_i} f(\mathbf{y})$
\item $f_i \in L^2$ and $f_i$ is absolutely continuous on its support, with
  a.e.\ partial derivatives $f_{ik}=\frac{\partial}{\partial y_k}
  f_i(\mathbf{y})$
\item $f_{ik} \in L^2$ for all $i,k$.
\end{enumerate}
We will presume below that $f$ has a bounded domain, but we do not list that as
a separate assumption, for two reasons.  First, even if $f$ has unbounded
support, we can always smoothly and invertibly map $\mathbb{R}^q$ into (say)
$[0,1]^q$, without disturbing any of the assumptions above, and without
changing $\beta(a)$, since total variation distance is invariant under
invertible transformations.  Second, the unbounded-domain case could be handled
by, basically, a remainder argument: the support of the histogram density
estimate is effectively set by the empirical range of the $\mathbf{X}_1^n$,
and, with high and growing probability, this includes the overwhelming majority
of the $f$ probability-mass.  Following this through would however needlessly
complicate our proofs.

Under these conditions, we can state the two main results of this paper.
Our first theorem in this section establishes consistency of
$\widehat{\beta}^{d_n}(a)$ as an estimator of $\beta(a)$ for all $\beta$-mixing
processes.
\begin{theorem}
  \label{thm:two}
  Let $\mathbf{X}_1^n$ be a sample from an arbitrary $\beta$-mixing process
  satisfying the conditions above.
  Provided that $nh_n^{d_n} \rightarrow \infty$, $d_nh_n
  \rightarrow 0$, $d_n \rightarrow \infty$, and $h_n\rightarrow 0$ as
  $n\rightarrow \infty$, then for any $\epsilon>0$,
  $$\lim_{n\rightarrow\infty}\P\left(\left|\widehat{\beta}^{d_n}(a)-\beta(a)\right|>\epsilon\right)=0.$$
\end{theorem}
In this general case, we need doubly asymptotic results about the
histogram estimator, that is, the histogram estimators require
shrinking bin widths in increasingly higher dimensions. In
\autoref{lem:optimalrates}, we give appropriate rates for $h_n$ and $d_n$
to achieve the optimal rate of convergence for the
estimation error. Of course, with discrete data or Markov models, we
may not need doubly asymptotic results since either the maximum number
of bins or the memory length of the process is fixed.

For a Markov process of order $d$ or less, $\beta^d(a)=\beta(a)$. In this case,
we can give the convergence rate of our estimator.
\begin{theorem}
  \label{thm:markov-rate}
  Let $\mathbf{X}_1^n$ be a sample from a  Markov process of order no
  larger than $d$. Then, taking the bandwidths to be
  $h_n=O((W(n)/n)^{2d/(2d+1)})$ for $\widehat{f}^d$ and
  $h_n=O((W(n)/n)^{4d/(4d+1)})$ for $\widehat{f}^{2d}$,
  \begin{equation}
    \E[|\widehat{\beta}^d(a)-\beta(a)|] = O\left(\sqrt{\frac{W(n)}{n}}\right).
  \end{equation}
\end{theorem}
Here, $W(n)$ is the Lambert $W$ function, i.e., the (multivalued) inverse of
$g(w) = w\exp\{w\}$ \citep{CorlessGonnet1996}. As $O(W(n))$ is bigger than
$O(\log\log n)$ but smaller than $O(\log n)$, our estimator attains nearly
parametric convergence rates when the data come from a Markov process of order
$\leq d$. Likewise, if we were interested in estimating only the finite
dimensional mixing coefficients $\beta^d(a)$ rather than $\beta(a)$,
\autoref{thm:markov-rate} gives the rate of convergence.

The proof of these two theorems requires showing the $L^1$ convergence of the
histogram density estimator with $\beta$-mixing data. We present this
result in \autoref{sec:conv-hist}. First, we discuss some important
details regarding these theorems and provide a method for choosing the
number of bins in 
the histogram.

\paragraph{Remarks on Discrete-Valued Processes}
\label{about-discrete-processes}

An important special case is that of discrete-valued $\beta$-mixing processes,
including Markov chains in the strict sense.  The symbols of any finite
alphabet can be represented by points in $\mathbb{R}$, and $\beta$-mixing
coefficients are invariant to the choice of representation.  Of course, the
resulting distributions in $\mathbb{R}^q$ will be mixtures of delta functions,
and so not absolutely continuous.  However, for a finite number of points,
there will exist a maximum bin-width below which each bin of the histogram will
contain at most one positive-probability point.  While the corresponding
histogram density estimator is always absolutely continuous, it is easily seen that below
this bin-width, the histogram estimator has the same $\beta$-mixing
coefficient as the true distribution.  Moreover, the errors of estimation and
approximation dealt with in the proofs of our theorems and lemmas are, if
anything, even smaller for finite-alphabet processes, making our results
somewhat conservative.

\paragraph{Remarks on the interaction between $q$, $d_n$, and $h_n$}
\label{about-dimensionality}

The interaction between the dimension of the data and the bandwidth of the
histogram (equivalently the number of bins used) is important for
applications. However, we use $d_n$ to represent more than the ``dimension'' of
the dataset: $d_n$ is the product of the dimension of the range of
$\mathbf{X}$, $q$, and the length of the Markov 
approximation. For example, suppose that the data consist of $q$
time-series (we will use a data set  with $q=6$ in \autoref{sec:examples}) which is
known to be first-order Markov. Then, $d=d_n=q\times 1$ for all $n$ is
sufficient for our estimator to achieve the convergence rate specified in
\autoref{thm:markov-rate} provided $h_n$ is chosen appropriately. However, if
this same dataset is non-Markovian but still $\beta$-mixing, then in order to
estimate $\beta(a)$ consistently, we must use successively larger Markov
approximations as $n\rightarrow\infty$. This means taking $d_n=q \times
\gamma(n)$ for an increasing function $\gamma$. Thus, even though the
data are
of fixed dimension $q$, the dimension over which the histogram is constructed
must increase to infinity to give estimation consistency as in
\autoref{thm:two}. \autoref{lem:optimalrates} shows that if
$\gamma(n)\sim\exp\{W(\log n)\}$, then there is a polynomial rate for the
bandwidth which satisfies the conditions of
\autoref{thm:two}. Of course for a
fixed dataset application, one must choose the bandwidth and potentially the
Markov approximation. One could fix the Markov approximation and use
cross-validation for the bandwidth selection, but we have found that this
procedure tends to choose bandwidths which are too small, resulting in a
positive estimation bias. In the remainder of this paper, we suppress
$q$ and work directly with $d_n$. In the next section, we present a procedure for
choosing the bandwidth for a fixed dimension $d$.

\subsection{Choosing the bandwidth}
\label{sec:choosing-bandwidth}

We need some way to pick the bandwidth of our histograms, or, equivalently, the
number of bins.  If we were doing density estimation for its own sake, the
natural thing to do would be some sort of cross-validation with a loss based on
the density.  However, we do not really care about the density.  Instead, we
suggest an approach which might be called ``calibration with surrogate
data''.\footnote{``Surrogate data'' methods are used extensively in nonlinear
  time series analysis for hypothesis testing, especially testing the
  hypothesis that there is some nonlinear deterministic structure
  \citep{KantzSchreiber2004}.  Note that we are using the word ``calibration''
  here in the sense in which measuring instruments are calibrated against
  standards, not the sense in which it is used in evaluating probabilistic
  forecasts \citep{GneitingBalabdaoui2007}, or the estimation
  technique from econometrics \citep{HansenHeckman1996}.} In
outline, we construct an artificial stochastic process which shares many
distributional features with the data, but where $\beta$ is known exactly.
This lets us see which bandwidth leads to the most accurate estimation of the
reference value of $\beta$, and this is at least a reasonable guess at the
appropriate bandwidth on the data.  To flesh this out, we first describe the
construction of the surrogate process with known $\beta$, and then the full
bandwidth-selection procedure.

We regard $d$, the order of the Markov approximation, as fixed, and note that, by
our error analysis in \autoref{thm:markov-rate}, the lag $a$ should not affect
the appropriate bandwidth.\footnote{Apart from leading to slight changes to the
  effective $n$.}  We thus proceed to construct a process where, for a given
$d$, the mixing coefficient $\beta^d(1)$ has a known value and then try to
optimize the variance of our estimator.

To generate the surrogate process, we sample blocks of length $d$ from
the data $\mathbf{X}_{1:n}$. We start with a random $d$-block $Z_1$ then
repeat that block with probability $1/2$ and resample a new $d$-block
with the remaining probability. We continue this process $M=n/d$ times
(rounding up or down as desired) until we have
a new sequence $\mathbf{Y}$ of length $Md$. Notice that the
$\mathbf{Y}$ process has the same marginal distribution as the
empirical marginal distribution of $\mathbf{X}$. Its higher-dimensional marginal
distributions are not guaranteed to match those of the data, because of the
abrupt change from one block to the next, and because of the random repetition of
blocks.  However, if $d \ll n$, the $d$-dimensional marginals should be close.
Based on this intuition, we present \autoref{alg:calibrate}
for choosing the bandwith in the histograms for continuous data
$\mathbf{X}_{1:n}$. We then prove two results justifying its use.
\begin{algorithm}
\SetKwInOut{Input}{input}\SetKwInOut{Output}{output}
  \Input{A timeseries $\mathbf{X}_{1:n}$; a finite approximation length
    $d$; desired number of replications $K$; candidate bandwidths
    $h=\{h_1,\ldots,h_H\}$}
  \Output{A bandwidth $h$}
  \BlankLine
  $M \leftarrow \lfloor n/d \rfloor$\;
  Calculate the $d$-dimensional histogram of $\mathbf{X}_{1:n}$,
  $[\widehat{p}_h^1,\ldots,\widehat{p}_h^J]$ for each $h$\;
  Estimate $\widehat{p}_h= \sum_{j=1}^J (\widehat{p}_h^j)^2$\;
  Calculate $\kappa$ as in \autoref{prop:calibrateTruth}\;
  \For{$k = 1$ \KwTo $K$}{
    \emph{generate a new series} $\mathbf{Y}^{(k)}$\;
    \For{$m=1$ \KwTo $M$}{
      Draw $U$ standard uniform\;
      \eIf{$m=1$ {\bf or} $U>1/2$}{
        Draw a random index $i \in \{1,\ldots, n-d+1\}$\;
        Set $Z^{(m)}_{1:d}\leftarrow\mathbf{X}_{i:i+d}$ and append this to $\mathbf{Y}^{(k)}$\;
      }{Set $Z^{(m)}_{1:d}=Z^{m-1}_{1:d}$ and append this to $\mathbf{Y}^{(k)}$\;}
    }
    \emph{estimate the mixing coefficient} $\widehat{\beta}^d_{(k)}(1)$
    for each $h$\;
  }
  \emph{Return the $h$ which minimizes the estimated
    variance} $\sum_{k=1}^K
  \left[\widehat{\beta}^d_{(k)}(1)-0.5(1-\kappa)(1-\widehat{p}_h)\right]^2.$ 
  \caption{Method to choose the bandwidth $h$ via calibration with
    surrogate data}
  \label{alg:calibrate}
\end{algorithm}

The first result presents the exact mixing coefficient for $\mathbf{Y}$.
\begin{proposition}
\label{prop:calibrateTruth}
Suppose that the marginal density $f$ of $\mathbf{X}$
is absolutely continuous. Fix $d$, and let $u$ be the set of unique
length-$d$ sequences appearing
in $\mathbf{X}_{1:n}$, where sequence $w \in u$ appears $n_w$ times.  Set
\[
\kappa = \sum_{w \in u}{\left(\frac{n_w}{n-d+1}\right)^2}
\]
Then for $\mathbf{Y}$ constructed as in \autoref{alg:calibrate},
\[
\beta^{d}(1) = 0.5(1-\kappa).
\]
\end{proposition}

\begin{proof}
First, let $Q=\P_{[d]} \otimes\P_{[d]}$ be the product measure of
$d$-blocks and call $P$ the joint 
distribution of $2d$-blocks associated to a hypothetical, infinite sequence
$\mathbf{Y}$ generated, say, by draws of $d$-blocks from the distribution of
$\mathbf{X}$ rather than its empirical counterpart. The total variation distance between the joint
distribution of two identical copies of the same block, and the joint distribution of two
independent blocks, is therefore $1$ (since the $P$ measure of this
set is 0).\footnote{Requiring any coordinate to be shared between two $d$-vectors $B$ and $C$
  forces the $(B,C)$ joint distribution to put probability 1 on a
  lower-dimensional subspace of the product space, which would have measure 0
  under the product measure, leading to a total variation distance of
  1 from independence.}  By the same reasoning, two blocks
which share some coordinates (even if not in the same positions within a block)
have a TV distance of $1$ from independence. For the diagonal $D$, we
therefore have that $P(D)=1/2$ while $Q(D)=0$. Thus, the total
variation between $P$ and $Q$ is at least $1/2$. To  show that it is
no more than $1/2$, suppose that there was another set $A$ where
$|P(A) - Q(A)| > 1/2$.  Without loss of generality, say $P(A) > Q(A)$.
(If the inequality went the other way, use $A^c$.)  Then $P(A) >
1/2 + Q(A)$, so $A$ must intersect the diagonal $D$; let $A = (A \cap D ) \cup
(A \cap D^c) = B \cup C$.  As disjoint sets, $P(A) = P(B) + P(C)$,
likewise for $Q$,  so $P(A) - Q(A) = P(B) - Q(B) + P(C) - Q(C)$, but $P(C)
=0.5 Q(C)$ and $Q(B) = 0$, thus $P(A) - Q(A) = P(B) - 0.5Q(C)$.    But
$P(B) \leq P(D) = 1/2$, and $Q(C) \geq 0$, so $P(A) - Q(A) \leq 1/2$.
Therefore, the total variation distance between $P$ and $Q$ is $1/2$.

Now, observe that $\kappa$ is the probability that two
independently drawn blocks in $\mathbf{Y}$ will, by chance, happen to be equal.
Thus, $\kappa$ simplifies to $1/(n-d+1)$ when all the blocks are distinct, as they
ought to be when the generating process has an absolutely continuous
distribution.
The factor $1-\kappa$ corrects for the fact that the empirical distributions
we are using put probability $\kappa$ on the low-dimensional
subspace $D$.
\end{proof}

With this result in hand, we will resample many time-series $\mathbf{Y}$ from
our data and choose the bandwidth by minimizing the variance over some grid of
$h$ values (equivalently number of bins). We do not minimize the mean squared
error, because the bias of the estimator depends on the mixing
coefficient. Minimizing the bias in an attempt to estimate a mixing coefficient
near $1/2$ may result in badly biased estimates of coefficients near zero. Our
next result calculates the expectation of this estimator and therefore its
bias.

\begin{proposition}
  \label{prop:calibrateExpect}
  Suppose that the marginal density $f^d$ of $\mathbf{X}_{[d]}$
  is absolutely continuous.
  The expected value of $\widehat{\beta}^d(1)$ based on
  $\mathbf{Y}$ is given by
  \[
  \E\left[\widehat{\beta}^d(1)\right] =
  0.5(1-p_h)(1-\kappa),
  \]
  where $p_h=\sum_{j=1}^J (\int _{B_j} f^d)^2$ and $\{B_j\}_{j=1}^J$ are
  the bins for a histogram with bandwidth $h$.
\end{proposition}
\begin{proof}
  By discretizing into histograms, the diagonal is no longer a measure
  0 set, and in fact contains more mass than $1/2$.
  By construction, the product distribution $Q$ puts mass
  $p_h$ on the discretized diagonal $D_h$ while the joint distribution, $P$,
  puts mass $0.5(1 + p_h^2)$ on the discretized diagonal. Therefore,
  under the histogram with bandwidth $h$, the total variation
  distance between $Q$ and $P$ is given by
  \begin{align*}
    \sup_A |Q(A)-P(A)| &= \sup_{D_h,D_h^c}|Q(D_h)-P(D_h)|\\
                       &= |p_h^2-0.5(1+p_h^2)| \vee |(1-p_h^2) -
                         0.5(1+(1-p_h^2))|\\
                       &= 0.5|p_h^2-1| \vee 0.5|1-p_h^2|\\
                       &=0.5(1-p_h^2).
  \end{align*}
\end{proof}

\section{$L^1$ convergence of histograms}
\label{sec:conv-hist}

Convergence of density estimators is thoroughly studied in the statistics and
machine learning literature. Early papers on the $L^\infty$ convergence of
kernel density estimators (KDEs) include
\citep{Woodroofe1967,BickelRosenblatt1973,Silverman1978};
\citet{FreedmanDiaconis1981} look specifically at histogram estimators, and
\citet{Yu1993} considers the $L^\infty$ convergence of KDEs for $\beta$-mixing
data and shows that the optimal IID rates can be
attained. \citet{Tran1994} proves $L^2$ convergence for histograms
under $\alpha$- and $\beta$-mixing.
\citet{DevroyeGyorfi1985} argue that $L^1$ is a more appropriate metric for
studying density estimation, and \citet{Tran1989} proves $L^1$ consistency of
KDEs under $\alpha$- and $\beta$-mixing. As far as we are aware, ours is the
first proof of $L^1$ convergence for histograms under
$\beta$-mixing.

Our proof requires the method of blocking used in~\citet{Yu1993,Yu1994}
following earlier results such as \citet{Eberlein1984,VolkonskiiRozanov1959}
and going back to Bernstein.  The idea here is to translate IID results
directly to mixing sequences, with corrections that reflect the
$\beta$-coefficients and the length of the process.  To do this, one creates an
imaginary sequence of independent blocks of data from the original dependent
sequence.  Ordinary IID results apply to the imaginary sequence, which also
approximates the actual dependent sequence, to within a known tolerance.

Consider a sample $\mathbf{X}_{1:n}$ from a stationary $\beta$-mixing sequence
with density $f$.  Let $m$ and $\mu$ be positive integers such that
$2m\mu=n$.  Now imagine dividing $\mathbf{X}_{1:n}$ into $2\mu$ blocks,
each of length $m$. Identify the blocks as follows:
\begin{align*}
  U_j &= \{X_i: 2(j-1)m + 1 \leq i \leq (2j-1)m\},\\
  V_j &= \{X_i : (2j-1)m + 1 \leq i \leq 2jm\},
\end{align*}
for $1\leq j \leq \mu$.  Let $\mathbf{U}$ be the entire sequence of odd
blocks $\{U_j\}_{j=1}^{\mu}$, and let $\mathbf{V}$ be the sequence of even
blocks $\{V_j\}_{j=1}^{\mu}$. A visual representation is shown in
\autoref{fig:blocking}. Finally, let $\widetilde{\mathbf{U}}$ be a sequence of
blocks which are independent of $\mathbf{X}_{1:n}$ but such that each block has
the same distribution as a block from the original sequence. That is, construct
$\widetilde{U}_j$ such that
\begin{align}
  \label{eq:2}
  \mathcal{L}(\widetilde{U}_j) = \mathcal{L}(U_j) = \mathcal{L}(U_1),
\end{align}
where $\mathcal{L}(\cdot)$ means the probability law of the argument. The blocks
$\widetilde{\mathbf{U}}$ are now an IID block sequence, in that for
integers $i,j\leq 2\mu$, $i\neq j$, $\widetilde{U}_i \indep \widetilde{U}_j$ so standard
results about IID random variables can be applied to these blocks.
(See \citep{Yu1994} for a more rigorous analysis of blocking.) We now
state the main result of this section.

\begin{figure}
  \centering
  \begin{tikzpicture}[scale=.5]
    \draw [fill=greenmain,ultra thick] (0,0) rectangle (2,1);
    \draw [fill=orangemain,ultra thick] (2,0) rectangle (4,1);
    \draw [fill=greenmain,ultra thick] (4,0) rectangle (6,1);
    \draw [fill=orangemain,ultra thick] (6,0) rectangle (8,1);
    \draw [fill] (8.5,.5) circle [radius=.1];
    \draw [fill] (9,.5) circle [radius=.1];
    \draw [fill] (9.5,.5) circle [radius=.1];
    \draw [fill=greenmain,ultra thick] (10,0) rectangle (12,1);
    \draw [fill=orangemain,ultra thick] (12,0) rectangle (14,1);
    \draw [fill] (14.5,.5) circle [radius=.1];
    \draw [fill] (15,.5) circle [radius=.1];
    \draw [fill] (15.5,.5) circle [radius=.1];
    \draw [fill=greenmain,ultra thick] (16,0) rectangle (18,1);
    \draw [fill=orangemain,ultra thick] (18,0) rectangle (20,1);
    \node at (1,.5) {$U_1$};
    \node at (3,.5) {$V_1$};
    \node at (5,.5) {$U_2$};
    \node at (7,.5) {$V_2$};
    \node at (11,.5) {$U_j$};
    \node at (13,.5) {$V_j$};
    \node at (17,.5) {$U_{\mu}$};
    \node at (19,.5) {$V_{\mu}$};
    \footnotesize
    \node at (11,-.5) {$\leftarrow m \rightarrow$};
    \node at (13,1.5) {$\leftarrow m \rightarrow$};
  \end{tikzpicture}
  \caption{The blocking procedure divides
    $\mathbf{X}_{1:n}$ into $2\mu$ alternating blocks $U_j$
    (orange) and $V_j$ (green)
    each of length $m$.}
  \label{fig:blocking}
\end{figure}
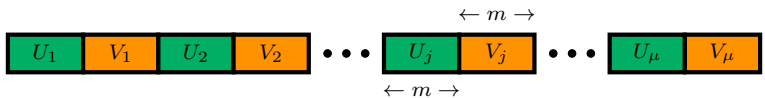

\begin{theorem}
  \label{thm:one}
  Let
  \begin{equation}
    \label{eq:hist}
    \widehat{f}(x) := \frac{1}{nh^d} \sum_{i=1}^n I(X_i \in B(x))
  \end{equation}
  be a histogram density estimator based on a 
  sample $\mathbf{X}_{1:n}$ from a $\beta$-mixing sequence with
  stationary density $f$, then for all
  $\epsilon>\E\left[\int |\widehat{f}-f|\right]$, and any natural
  numbers $m$ and $\mu$ such that $2m\mu\leq n$,
  \begin{align}
    \label{eq:5}
    \P\left(\int |\widehat{f}-f| > \epsilon\right) &\leq
    2\exp\left\{-\frac{\mu\epsilon_1^{2}}{2}\right\}+ 2\mu\beta(m)
  \end{align}
  where $\epsilon_1 = \epsilon-\E\left[\int |\widehat{f}-f|\right]$.
\end{theorem}

This theorem demonstrates a clear tradeoff between the mixing behavior of
the stochastic process and the ability to concentrate the estimator
close to the truth. For arbitrary $\beta$-mixing processes, we cannot
actually say much about the quality of this estimator other than that
given enough data, it will eventually do well. For this reason, one
generally assumes that the mixing coefficients
$\beta(a)$ display particular asymptotic behaviors like exponential or
polynomial decay.

To prove \autoref{thm:one}, we use the blocking method of~\citep{Yu1994} to
transform the dependent $\beta$-mixing sequence into a sequence of nearly
independent blocks. We then apply McDiarmid's inequality to the blocks
to derive asymptotics in the bandwidth of the histogram as well as the
dimension of the target density. For completeness, we state a version
of Yu's
blocking result and McDiarmid's inequality before proving the doubly
asymptotic histogram convergence for IID data. Combining these lemmas
allows us to prove concentration results for histograms based on
$\beta$-mixing inputs.

\begin{lemma}[Lemma 4.1 in~\citep{Yu1994}]
  \label{lem:yu}
  Let $\phi$ be an event with respect to the block sequence
  $\mathbf{U}$. Then,
  \begin{equation}
    \label{eq:8}
    |\P[\phi] - \widetilde{\P}[\phi]| \leq \mu\beta(m),
  \end{equation}
  where the first probability $\P$ is with respect to the dependent block
  sequence, $\mathbf{U}$, and
  $\widetilde{\P}$ is the $\mu$-fold product measure created with
  the marginal distribution of each block $\mathbf{U}$,
  i.e. $\widetilde{\P} = (\P_{[m]})^{\mu}$.
\end{lemma}
This lemma essentially gives a method of applying IID results to
$\beta$-mixing data. Because the dependence decays as we increase the
separation between blocks, widely spaced blocks are nearly independent
of each other. In particular, the difference between probabilities of
events generated by
these nearly independent blocks and probabilities with respect to blocks which are
actually independent can be controlled by the $\beta$-mixing
coefficient.

\begin{lemma}[McDiarmid Inequality~\citep{McDiarmid1989}]
  \label{lem:mcdiarmid}
  Let $X_1,\ldots,X_n$ be independent random variables, with $X_i$
  taking values in a set $A_i$ for each $i$. Suppose that the
  measurable function $f:\prod A_i \rightarrow \R$ satisfies
  \[
  |f(\mathbf{x})-f(\mathbf{x}')| \leq c_i
  \]
  whenever the vectors $\mathbf{x}$ and $\mathbf{x}'$ differ only in
  the $i^{th}$ coordinate. Then for any $\epsilon>0$,
  \[
  \P(f - \E f > \epsilon) \leq \exp\left\{-\frac{2\epsilon^2}{\sum
      c_i^2} \right\}.
  \]
\end{lemma}

Since we will need  the dimension of the histograms to grow with
$n$, we prove the following lemma which provides the doubly asymptotic convergence of the
histogram estimator for IID data. It differs from standard histogram
convergence results in the bias calculation. In this case we need to
be more careful about the interaction between $d$ and $h_n$.
\begin{lemma}
  \label{lem:three}
  For an IID sample $X_1,\ldots,X_n$ from some density $f$ on $\R^d$,
  \begin{align}
    \label{eq:6}
    \E \int |\widehat{f}-\E\widehat{f}|dx  &= O\left(1/\sqrt{nh_n^d}\right)\\
    \int |\E\widehat{f} - f|dx &= O(dh_n)+O(d^2h_n^2),
  \end{align}
  where $\widehat{f}$ is the histogram estimate using a grid with
  sides of length $h_n$.
\end{lemma}
\begin{proof}[Proof of \autoref{lem:three}]
  Let $p_j$ be the probability of falling into the $j^{th}$ bin
  $B_j$. Denote the total number of bins by $J=h_n^{-d}$. Then,
  \begin{align*}
    \E \int |\widehat{f} - \E\widehat{f}| &= h_n^d \sum_{j=1}^J
    \E\left| \frac{1}{nh_n^d} \sum_{i=1}^n I(X_i \in B_j) -
      \frac{p_j}{h^d} \right|\\
    &\leq h^d_n \sum_{j=1}^J \frac{1}{nh_n^d} \sqrt{\V\left[\sum_{i=1}^n I(X_i \in B_j)\right]}
    = h^d_n \sum_{j=1}^J \frac{1}{nh_n^d} \sqrt{np_j(1-p_j)}\\
    &= \frac{1}{\sqrt{n}}\sum_{j=1}^J \sqrt{p_j(1-p_j)}
    = O(n^{-1/2})O(h_n^{-d/2}) = O\left(1/\sqrt{nh_n^d}\right).
  \end{align*}
  Using a Taylor expansion
  \begin{align*}
   f(\mathbf{x}) &= f(\mathbf{c}) + \sum_{i=1}^d
    (x_i-c_i)f_i(\mathbf{c})+ O(d^2h_n^2),
  \end{align*}
  where
  $f_i(\mathbf{y}) = \frac{\partial}{\partial y_i} f(\mathbf{y})$. Therefore, $p_j$
  is given by
  \begin{equation*}p_j = \int_{B_j} f(x)dx = h_n^df(c) +
    O(d^2h_n^{d+2})
  \end{equation*}
  since the integral of the second term over the bin
  is zero. This means that for the $j^{th}$ bin,
  \begin{align*}
    \E\widehat{f}_n(x)-f(x) &=
    \frac{p_j}{h_n^d} - f(x) = -\sum_{i=1}^d
    (x_i-c_i)f_i(\mathbf{c}) + O(d^2h_n^2).
  \end{align*}
  Therefore,
  \begin{align*}
    \displaystyle \int_{B_j} \left| \E\widehat{f}_n(x)-f(x)
  \right|
    &= \int_{B_j}
    \left| -\sum_{i=1}^d (x_i - c_i)f_i(\mathbf{c}) + O(d^2h_n^2)
    \right|\\
    &\leq \int_{B_j} \left| -\sum_{i=1}^d (x_i - c_i)f_i(\mathbf{c})
    \right| + \int_{B_j} O(d^2h^2)\\
   &= \int_{B_j} \left| \sum_{i=1}^d (x_i - c_i) f_i(\mathbf{c})
    \right| + O(d^2h_n^{2+d})\\
    &=  O(dh_n^{d+1}) + O(d^2h_n^{2+d})
  \end{align*}
Since each bin is bounded, we can sum over all $J$ bins. The number of
bins is $J=h_n^{-d}$ by definition, so
\begin{align*}
  \int|\E\widehat{f}_n(x)-f(x)|dx
  &= O(h_n^{-d})\left( O(dh_n^{d+1}) +  O(d^2h_n^{2+d})\right)
   = O(dh_n) + O(d^2h_n^2).
\end{align*}
\end{proof}

The dimension of the target density
is analogous to the order of the Markov approximation. Therefore, the
convergence rates we give are asymptotic in the bandwidth $h_n$ which
shrinks as $n$ increases, but also in the dimension $d$ which
increases with $n$. Even under these asymptotics, histogram estimation
in this sense is not a high dimensional problem. The dimension of the
target density considered here is on the order of $\exp\{W(\log n)\}$,
a rate somewhere between $\log n$ and $\log\log n$.

We can combine the above lemmas to prove the main result of this
section. Essentially, we use \autoref{lem:yu} to transform the problem
from one about dependent data points to one involving independent
blocks, we then apply \autoref{lem:mcdiarmid} to the blocks to get
one-sided concentration inequalities, and finally, we use
\autoref{lem:three} to ensure that certain expectations are bounded.

\begin{proof}[Proof of \autoref{thm:one}]
  Let $g$ be the $L^1$ loss of the histogram estimator, $g=\int
  |f-\widehat{f}|$ where $\widehat{f}$ is defined in \eqref{eq:hist}. Let $\widehat{f}_{\mathbf{U}}$,
  $\widehat{f}_{\mathbf{V}}$, and $\widehat{f}_{\widetilde{\mathbf{U}}}$ be
  histograms based on the block sequences $\mathbf{U}$, $\mathbf{V}$,
  and $\widetilde{\mathbf{U}}$ respectively. Then
  \begin{align*}
    \widehat{f}(x) &= \frac{1}{nh^d}\sum_{i=1}^n I(X_i \in B(x))\\ &=
    \frac{1}{nh^d} \sum_{j=1}^{\mu}\sum_{i=2(j-1)m+1}^{(2j-1)m}
    I(X_i \in B(x)) + \frac{1}{nh^d} \sum_{j=1}^{\mu}\sum_{i=(2j-1)m+1}^{2jm} I(X_i \in B(x))\\&=
    \frac{1}{2}(\widehat{f}_{\mathbf{U}} + \widehat{f}_{\mathbf{V}}).
  \end{align*}
  Now,
  \begin{align}
    \P(g>\epsilon) &= \P\left(\int |f-\widehat{f}| >
      \epsilon\right)
    = \P\left(\int \left|\frac{f-\widehat{f}_{\mathbf{U}}}{2} +
          \frac{f-\widehat{f}_{\mathbf{V}}}{2} \right| >
      \epsilon\right)\notag\\
    &\leq \P\left( \frac{1}{2} \int |f-\widehat{f}_{\mathbf{U}}| +
      \frac{1}{2} \int|f-\widehat{f}_{\mathbf{V}}| >
      \epsilon\right)
    = \P(g_{\mathbf{U}} + g_{\mathbf{V}} > 2\epsilon) \notag\\
    &\leq \P(g_{\mathbf{U}} >\epsilon) + \P(g_{\mathbf{V}}>\epsilon)
    = 2\P(g_{\mathbf{U}} - \E[g_{\mathbf{U}}] > \epsilon-\E[g_{\mathbf{U}}]) \notag\\
    &= 2\P(g_{\mathbf{U}} - \E[g_{\widetilde{\mathbf{U}}}] >
    \epsilon-\E[g_{\widetilde{\mathbf{U}}}])
    = 2\P(g_{\mathbf{U}} - \E[g_{\widetilde{\mathbf{U}}}] > \epsilon_1)\notag,
  \end{align}
  where the equality in the last line (using
  $\E[g_{\mathbf{U}}]=\E[g_{\widetilde{\mathbf{U}}}]$) is implicit in
  the construction of $\widetilde{\mathbf{U}}$ from
   \eqref{eq:2} and $\epsilon_1 = \epsilon-\E[g_{\widetilde{\mathbf{U}}}]$. Here,
  \[
  \E[g_{\widetilde{\mathbf{U}}}] \leq \widetilde{\E} \int
  |\widehat{f}_{\widetilde{\mathbf{U}}}-\widetilde{\E}\widehat{f}_{\widetilde{\mathbf{U}}}|dx +
  \int |\tilde{\E}\widehat{f}_{\widetilde{\mathbf{U}}} - f|dx,
  \]
  so by \autoref{lem:three}, as long as for $\mu\rightarrow\infty$, $h_n
  \downarrow 0$ and $\mu h_n^d \rightarrow \infty$, then for all
  $\epsilon$ there exists $n_0(\epsilon)$ such that for all $n>n_0(\epsilon)$,
  $\epsilon>\E[g]=\E[g_{\widetilde{\mathbf{U}}}]$. Now
  applying \autoref{lem:yu} to
  the event $\{g_{\mathbf{U}} - \E[g_{\widetilde{\mathbf{U}}}] > \epsilon_1\}$
  gives
  \begin{align*}
    2\P(g_{\mathbf{U}} - \E[g_{\widetilde{\mathbf{U}}}] > \epsilon_1) &\leq
    2\P(g_{\widetilde{\mathbf{U}}} - \E[g_{\widetilde{\mathbf{U}}}] >
    \epsilon_1)+ 2\mu\beta(m)
  \end{align*}
  where the probability on the right is for the $\sigma$-field
  generated by the independent block sequence $\widetilde{\mathbf{U}}$. Since
  these blocks are independent, showing that $g_{\widetilde{\mathbf{U}}}$
  satisfies the bounded differences requirement allows for the
  application of \autoref{lem:mcdiarmid} to the blocks. For any two block
  sequences $z_1,\ldots,z_{\mu}$ and ${z}'_1,\ldots,{z}'_{\mu}$ with
  $z_\ell={z}'_\ell$ for all $\ell\neq j$, then
  \begin{align*}
    \lefteqn{\left| g_{\tilde{\mathbf{U}}}(z_1,\ldots,z_{\mu}) -
      g_{\tilde{\mathbf{U}}}({z}'_1,\ldots,{z}'_{\mu}) \right|}\\ & =
    \left|\int |\widehat{f}(y;z_1,\ldots,z_{\mu})-f(y)|dy
      -\int|\widehat{f}(y;z'_1,\ldots,z'_{\mu}) - f(y)|dy\right|\\
    &\leq \int |\widehat{f}(y;z_1,\ldots,z_{\mu}) -
    \widehat{f}(y;z'_1,\ldots,z'_{\mu})|dy\\
    &= \frac{2}{\mu h_n^d} h_n^d = \frac{2}{\mu}.
  \end{align*}
  Therefore,
  \begin{align*}
     \P(g>\epsilon) &\leq 2\P(g_{\tilde{\mathbf{U}}} - \E[g_{\tilde{\mathbf{U}}}] >
     \epsilon_1) + 2\mu\beta(m) \\
     & \leq 2\exp\left\{ -\frac{\mu\epsilon_1^2}{2}\right\} +
     2\mu\beta(m).
  \end{align*}
\end{proof}

\section{Proofs of results in \autoref{sec:assumpt-main-results}}
\label{sec:proof}

With
the structure from the previous section, we can state a concentration
inequality for $\widehat{\beta}^d(a)$.
\begin{lemma}
  \label{lem:main}
  Consider a sample $\mathbf{X}_{1:n}$ from a stationary
  $\beta$-mixing process. Let $\mu$ and $m$
  be positive integers such that $2\mu m\leq n$ and $\mu\geq d>0$.  Then
\begin{align*}
\P(|\widehat{\beta}^d(a) - \beta^d(a)|>\epsilon)
    & \leq
    2\exp\left\{-\frac{\mu \epsilon_1^2}{2}\right\}
    + 2\exp\left\{-\frac{\mu \epsilon_2^2}{2}\right\}
    + 4\mu\beta(m),
\end{align*}
  where $\epsilon_1 = \epsilon/2-\E\left[\int|\widehat{f}^d - f^d|\right]$ and
  $\epsilon_2 = \epsilon - \E\left[\int|\widehat{f}_a^{2d} - f_a^{2d}|\right]$.
\end{lemma}

The proof of \autoref{lem:main} relies on the triangle inequality and the
relationship between total variation distance and the $L^1$ distance
between densities.

\begin{proof}[Proof of \autoref{lem:main}]
  For any two probability measures $\nu$ and $\lambda$ defined on the same
  probability space with associated densities $f_\nu$ and $f_\lambda$
  with respect to some dominating measure $\pi$,
  \[
  \TV{\nu-\lambda} = \frac{1}{2}\int d(\pi)|f_\nu - f_\lambda|.
  \]
  Recall that  $\P_{[d]}$ is the $d$-dimensional stationary distribution of the
  $d^{th}$-order Markov approximation
  in the notation
  of \eqref{eq:2a}, and $\P_{[d],a}$ is the
  joint distribution of the bivariate random process created by the
  initial process and itself separated by $a$ time steps. By the
  triangle inequality, we can upper bound $\beta^d(a)$ for any $d=d_n$. Let
  $\widehat{\P}_{[d]}$ and $\widehat{\P}_{[d],a}$ be the distributions
  associated with histogram estimators $\widehat{f}^d$ and
  $\widehat{f}_a^{2d}$ respectively. Then,
  \begin{align*}
    \beta^d(a) &= \TV{\P_{[d]} \otimes \P_{[d]} -
                 \P_{[d],a }}\\
    &=\TV{ \P_{[d]} \otimes \P_{[d]} -\widehat{\P}_{[d]} \otimes \widehat{\P}_{[d]} +
        \widehat{\P}_{[d]}\otimes \widehat{\P}_{[d]} -
        \widehat{\P}_{[d],a} + \widehat{\P}_{[d],a}
        - \P_{[d],a } }\\
    &\leq \TV{\P_{[d]}\otimes \P_{[d]} -\widehat{\P}_{[d]} \otimes \widehat{\P}_{[d]}} +
    \TV{\widehat{\P}_{[d]}\otimes \widehat{\P}_{[d]} -
      \widehat{\P}_{[d],a}}+ \TV{ \widehat{\P}_{[d],a} - \P_{[d],a
      }}\\ 
    &\leq 2\TV{\P_{[d]} - \widehat{\P}_{[d]}} +
    \TV{\widehat{\P}_{[d]}\otimes \widehat{\P}_{[d]} - \widehat{\P}_{[d],a}}
    + \TV{ \widehat{\P}_{[d],a} - \P_{[d],a }}\\
    &= \int |f^d - \widehat{f}^d| +\frac{1}{2}\int |\widehat{f}^d\otimes
    \widehat{f}^d - \widehat{f}_a^{2d}|
    + \frac{1}{2}\int|f_a^{2d} -\widehat{f}_a^{2d}|
  \end{align*}
  where $\frac{1}{2}\int |\widehat{f}^d\otimes
    \widehat{f}^d - \widehat{f}_a^{2d}|$ is our estimator $\widehat{\beta}^d(a)$ and
  the remaining terms are the $L^1$ distance between a
  density estimator and the target density. Thus,
  \[
  \beta^d(a) - \widehat{\beta}^d(a) \leq \int |f^d - \widehat{f}^d| +
  \frac{1}{2}\int|f_a^{2d} -\widehat{f}_a^{2d}| .
  \]
  A similar argument starting from $\widehat{\beta}^d(a)= \TV{\widehat{\P}_{[d]}\otimes \widehat{\P}_{[d]} - \widehat{\P}_{[d],a}}$
  shows that
  \[
  \widehat{\beta}^d(a) - \beta^d(a) \leq \int |f^d - \widehat{f}^d| +
  \frac{1}{2}\int|f_a^{2d} -\widehat{f}_a^{2d}|,
  \]
  so we have that
  \[
  \left|\beta^d(a) - \widehat{\beta}^d(a)\right| \leq \int |f^d - \widehat{f}^d| +
  \frac{1}{2}\int|f_a^{2d} -\widehat{f}_a^{2d}| .
  \]
  Therefore,
  \begin{align*}
\P\left(\left|\beta^d(a) - \widehat{\beta}^d(a)\right| >
      \epsilon\right)
    & \leq \P\left( \int |f^d - \widehat{f}^d| +
      \frac{1}{2}\int|f_a^{2d} -\widehat{f}_a^{2d}|>\epsilon\right)\\
    & \leq \P\left( \int |f^d - \widehat{f}^d|>\frac{\epsilon}{2}\right) + \P\left(
      \frac{1}{2}\int|f_a^{2d} -\widehat{f}_a^{2d}|>\frac{\epsilon}{2}\right)\\
    &\leq 2\exp\left\{-\frac{\mu\epsilon_1^2}{2}\right\} +
    2\exp\left\{-\frac{\mu\epsilon_2^2}{2}\right\}\\
    &\quad+ 4\mu\beta(m),
  \end{align*}
  where $\epsilon_1 = \epsilon/2-\E\left[\int |\widehat{f}^d - f^d|\right]$
  and $\epsilon_2 = \epsilon-\E\left[\int |\widehat{f}_a^{2d} - f_a^{2d}|\right]$.
\end{proof}

\begin{proof}[Proof of \autoref{thm:markov-rate}]
  By \autoref{lem:main}, we have
  \begin{align*}
    \E[|\widehat{\beta}(a)-\beta(a)|] &= \int_0^1 d\epsilon\ 
    \P(|\widehat{\beta}^d(a) - \beta^d(a)|>\epsilon)\\
    &\leq \int_0^1 d\epsilon \left[2\exp\left\{-\frac{\mu \epsilon_1^2}{2}\right\}
    + 2\exp\left\{-\frac{\mu \epsilon_2^2}{2}\right\}
    + 4\mu\beta(m)\right]\\
    &= O(\mu^{-1/2}) + 4\mu\beta(m).
  \end{align*}
  To balance both terms, one needs $\beta(m)=O(\mu^{-3/2})$. Since
  $\beta(m)=O(\rho^{-m})$ for Markov processes, then taking
  $m=\frac{3}{2}\log_\rho \mu$ is sufficient. Now, solving
  \begin{equation*}
    n=2\mu \frac{3}{2}\frac{\log \mu}{\log \rho}
  \end{equation*}
  gives $\mu= O(n/W(n))$ giving the result.
\end{proof}

The proof of \autoref{thm:two} requires two steps which are given
in the following Lemmas. The first specifies the histogram
bandwidth $h_n$ and the rate at which
$d_n$ (the dimensionality of the target density) goes to infinity. If
the dimensionality of the target density were fixed, we
could achieve rates of convergence similar to those for histograms
based on IID inputs as shown in \autoref{thm:markov-rate}. However, we wish to allow the dimensionality
to grow with $n$, so the rates are much slower as shown in the
following lemma.
\begin{lemma}
  \label{lem:optimalrates}
  For the histogram estimator in \eqref{eq:estimator}, let
    $d_n \sim \exp\{W(\log n)\}$ and
    $h_n \sim n^{-k_n}$ with
  \begin{align*}
    k_n &= \frac{ W(\log n) + \frac{1}{2} \log n}{\log n \left( \frac{1}{2}
        \exp\{ W(\log n)\} + 1\right)}.
  \end{align*}
  Then, for all $\epsilon>0$,
  $\lim_{n\rightarrow\infty}\P(|\widehat{\beta}^{d_n}(a)-\beta^{d_n}(a)|>\epsilon)=0$.
\end{lemma}
\begin{proof}[Proof of \autoref{lem:optimalrates}]
  Let $h_n=n^{-k_n}$ for some $k_n$ to be determined. Then from
  \autoref{lem:three} we
  want $n^{-1/2}h_n^{-d_n/2} = n^{(k_nd_n-1)/2}\rightarrow 0$, $d_nh_n =
  d_nn^{-k}\rightarrow 0$, and $d_n^2h_n^2 = d_n^2n^{-2k}\rightarrow 0$ all as
  $n\rightarrow\infty$. Call these $A$, $B$, and $C$. Taking $A$ and $B$
  first gives
  \begin{align}
    n^{(k_nd_n-1)/2} &\sim d_nn^{-k_n}\nonumber\\
    \Rightarrow \frac{1}{2} (k_nd_n-1) \log n &\sim \log d_n-k_n\log n \nonumber\\
    \Rightarrow k_n\log n \left(\frac{1}{2}d_n+1\right) &\sim \log d_n +
    \frac{1}{2} \log n\nonumber\\
    \label{eq:3}
    \Rightarrow k_n &\sim \frac{\log d_n + \frac{1}{2}\log n}{\log n
      \left(\frac{1}{2} d_n + 1\right)}.
  \end{align}
  Similarly, combining $A$ and $C$ gives
  \begin{equation}
    \label{eq:4}
    k_n \sim \frac{2\log d_n + \frac{1}{2}\log n}{\log n
      \left(\frac{1}{2} d_n + 2\right)}.
  \end{equation}
  Equating~(\ref{eq:3}) and~(\ref{eq:4}) and solving for $d_n$ gives
  \[
  \Rightarrow d_n \sim \exp\left\{ W(\log n)\right\}
  \]
  where $W(\cdot)$ is the Lambert $W$ function. Plugging back
  into~(\ref{eq:3}) gives that
  \[
  h_n = n^{-k_n}
  \]
  where
 \[
 k_n = \frac{
   W(\log n) + \frac{1}{2}\log n} { \log n \left( \frac{1}{2}
     \exp\left\{ W(\log n)\right\} + 1\right)}.
 \]
\end{proof}
It is also necessary to show that as $d$ grows, we have the
nonstochastic convergence $\beta^d(a)\rightarrow
\beta(a)$. We now prove this result.
\begin{lemma}
  \label{lem:two2}
  $\beta^d(a)$ converges to $\beta(a)$ as $d\rightarrow\infty$.
\end{lemma}
\begin{proof}[Proof of \autoref{lem:two2}]
  We
  can rewrite \autoref{defn:beta-mix} as
  \begin{align*}
    \beta(a)=\sup_{C\in\sigma_{[\infty],a}}|\P_{[\infty],a}(C)-[\P_{-\infty:0}\otimes
    \P_{a:\infty}](C)|. 
  \end{align*}
  and
  $\beta^d(a)$ as
  \begin{equation}
    \beta^d(a)=\sup_{C\in\sigma_{[d],a}}|\P_{[\infty],a}(C)-[\P_{-\infty:0}\otimes
    \P_{a:\infty}](C)|
    \label{eq:11}
  \end{equation}
  As such $\beta^d(a)\leq\beta(a)$ for all $a$ and $d$. We can
  rewrite~(\ref{eq:11}) in terms of finite-dimensional marginals:
  \begin{equation*}
    \beta^d(a)=\sup_{C\in\sigma_{[d],a}}|\P_{[d],a} (C)-[\P_{-d+1:0}\otimes
    \P_{a:(a+d-1)}](C)|.
  \end{equation*}
  Because of
  the nested nature of these sigma-fields, we have
  \begin{equation*}
    \beta^{d_1}(a) \leq \beta^{d_2}(a) \leq \beta(a)
  \end{equation*}
  for all finite $d_1 \leq d_2$. Therefore, for fixed $a$,
  $\{\beta^d(a)\}_{d=1}^\infty$ is a monotone increasing sequence
  which is bounded above, and it converges
  to some limit $L\leq \beta(a)$. To show that $L=\beta(a)$ requires
  some additional steps.

   Let $R=\P_{[\infty],a}-[\P_{-\infty:0}\otimes
    \P_{a:\infty}]$, which is a signed
   measure on $\sigma$. Let $$R^d  = \P_{[d],a}-[\P_{-d+1:0}\otimes
    \P_{a:(a+d-1)}],$$ which is a signed measure on
  $\sigma_{[d],a}$. Decompose $R$ into positive and negative parts as
  $R=Q^+-Q^-$ and similarly for $R^d=Q^{+d}-Q^{-d}$. Notice that since
  $R^d$ is constructed using the marginals of $\P$, then $R(E) =
  R^d(E)$ for all $E\in\sigma_{[d],a}$. Now since $R$ is
  the difference of probability measures, we must have that
  \begin{align}
    0 &=R(\Omega)=Q^+(\Omega)-Q^-(\Omega)\nonumber\\
    &=Q^+(D)+Q^+(D^c) -Q^-(D)-Q^-(D^c) \label{eq:19}
  \end{align}
  for all $D\in \sigma$.

  Define $Q=Q^++Q^-$.  Let $\epsilon>0$. Let $C\in\sigma$ be such that
  \begin{equation}
    Q(C) =\beta(a)=Q^+(C)=Q^-(C^c).
    \label{eq:17}
  \end{equation}
  Such a set $C$ is guaranteed by the Hahn
  decomposition theorem (letting $C^*$ be a set which attains the
  supremum in~(\ref{eq:11}), we can throw away any subsets with negative
  $R$ measure) and~(\ref{eq:19}) assuming without loss of
  generality that $\P_{[\infty],a}
  (C)>[\P_{-\infty:0}\otimes \P_{a:\infty}](C)$. We can use the field 
  $\sigma_f=\bigcup_d \sigma_{[d],a}$ to
  approximate $\sigma$ in the sense that, for all $\epsilon$, we can
  find $A\in\sigma_f$ such that $Q(A\Delta C) < \epsilon/2$ (see
  Theorem D in~\citet[\S 13]{Halmos1974} or Lemma
  A.24 in~\citet{Schervish1995}). Now,
  \begin{align*}
    Q(A\Delta C) &= Q(A\cap C^c)+Q(C\cap A^c)\\
    &= Q^-(A \cap C^c) +Q^+(C\cap A^c)
  \end{align*}
  by~(\ref{eq:17}) since $A\cap C^c \subseteq C^c$ and $C\cap A^c
  \subseteq C$. Therefore, since $Q(A\Delta C) < \epsilon/2$, we have
  \begin{align}
    \label{eq:20}
    Q^-(A\cap C^c) &\leq \epsilon/2\\
    Q^+(A^c \cap C) & \leq \epsilon/2.\nonumber
  \end{align}
  Also,
  \begin{align*}
    Q(C) &= Q(A\cap C) + Q(A^c\cap C)\\
    &= Q^+(A\cap C) + Q^+(A^c\cap C)\\
    &\leq Q^+(A) +\epsilon/2
  \end{align*}
  since $A\cap C$ and $A^c\cap C$ are contained in $C$ and $A\cap C
  \subseteq A$. Therefore
  \[
  Q^+(A) \geq Q(C) - \epsilon/2.
  \]
  Similarly,
  \[
  Q^-(A) = Q^-(A\cap C) + Q^-(A\cap C^c) \leq 0 + \epsilon/2=\epsilon/2
  \]
  since $A\cap C \subseteq C$ and $Q^-(C)=0$ by~(\ref{eq:20}). Finally,
  \begin{align*}
    Q^{+d}(A) &\geq Q^{+d}(A) - Q^{-d}(A) = R^d(A)\\
    &= R(A) = Q^+(A) - Q^-(A)\\
    &\geq Q(C) - \epsilon/2-\epsilon/2 = Q(C)-\epsilon\\
    &= \beta(a) - \epsilon.
  \end{align*}
  And since $\beta^d(a) \geq Q^{+d}(A)$, we have that for all
  $\epsilon>0$ there exists $d$ such that for all $d_1>d$,
  \begin{align*}
    \beta^{d_1}(a) &\geq \beta^d(a) \geq Q^{+d}(A) \geq \beta(a)-\epsilon.
  \end{align*}
  Thus, we must have that $L=\beta(a)$, so that
  $\beta^d(a)\rightarrow\beta(a)$ as desired.
\end{proof}

\begin{proof}[Proof of \autoref{thm:two}]
By the triangle inequality,
  \[
  |\widehat{\beta}^{d_n}(a) - \beta(a)| \leq
  |\widehat{\beta}^{d_n}(a)-\beta^{d_n}(a)|   + |\beta^{d_n}(a) - \beta(a)|.
  \]
  The first term on the right
  is bounded by the result in \autoref{lem:main}, where we have
  shown that $d_n=O(\exp\{W(\log n)\})$ is slow enough for the histogram
  estimator to remain consistent.
That $\beta^{d_n}(a)
  \xrightarrow{d_n\rightarrow\infty} \beta(a)$ follows from
\autoref{lem:two2}.
\end{proof}

\section{Performance in examples}
\label{sec:examples}

To demonstrate the performance of our estimator, we examine
three simulated examples and an example using real data.

\subsection{Simulations}
\label{sec:simulations}

The first simulation is a
simple two-state Markov chain. Thus, its mixing rate is known, only
two bins are required in the histogram, and we can use $d=1$. The
second takes this Markov
chain as an unobserved input and outputs a non-Markovian binary
sequence which remains $\beta$-mixing, but we must now allow $d$ to
grow with $n$. Finally, we examine an
autoregressive model wherein we can again use $d=1$ as it is
Markovian, but there is an
uncountable state space.

\subsubsection{Markov process}
\label{sec:markov-process}

As shown in~\citep{Davydov1973}, homogeneous recurrent Markov chains
are geometrically $\beta$-mixing, i.e.~$\beta(a) = O(\rho^a)$ for some
$0\leq\rho<1$. In particular, if the Markov chain has stationary
distribution $\pi$ and $a$-step transition distribution $P^a$, then
\begin{equation}
  \label{eq:111}
  \beta(a) = \int \pi(dx) \TV{P^a(\cdot\ |\ x) - \pi(\cdot)}.
\end{equation}

Consider first the two-state Markov chain $S_t$ pictured in \autoref{fig:markov}.
\begin{figure}[t!]
  \centering
  \includegraphics{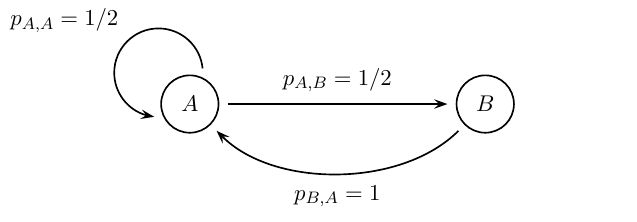}
  \caption{Two-state Markov chain $S_t$ used for simulations.}
  \label{fig:markov}
\end{figure}
By direct calculation using (\ref{eq:111}), the mixing coefficients for
this process are $\beta(a) =
\frac{4}{9}\left(\frac{1}{2}\right)^a$. We simulated chains of length
$n=1000$ from this Markov chain. \autoref{fig:latent} shows the
performance of the estimator based on 1000 replications. Here, we
have used two bins in all 
cases (as there are only two states), but we allow the Markov approximation to vary as
$d\in\{1,\ 2,\ 3,\ 4\}$, even though $d=1$ is exact. The estimator performs
well for $a\leq 5$, but begins to exhibit a positive bias as $a$
increases. This is because the estimator is nonnegative, whereas the
true mixing rates are quickly approaching zero. The upward bias is
exaggerated for larger $d$. This bias goes away as
$n\rightarrow\infty$. This is demonstrated in \autoref{fig:latentbig}
which uses $n=100,000$.
\begin{figure}[t!]
  \centering
  \includegraphics[width=.9\textwidth]{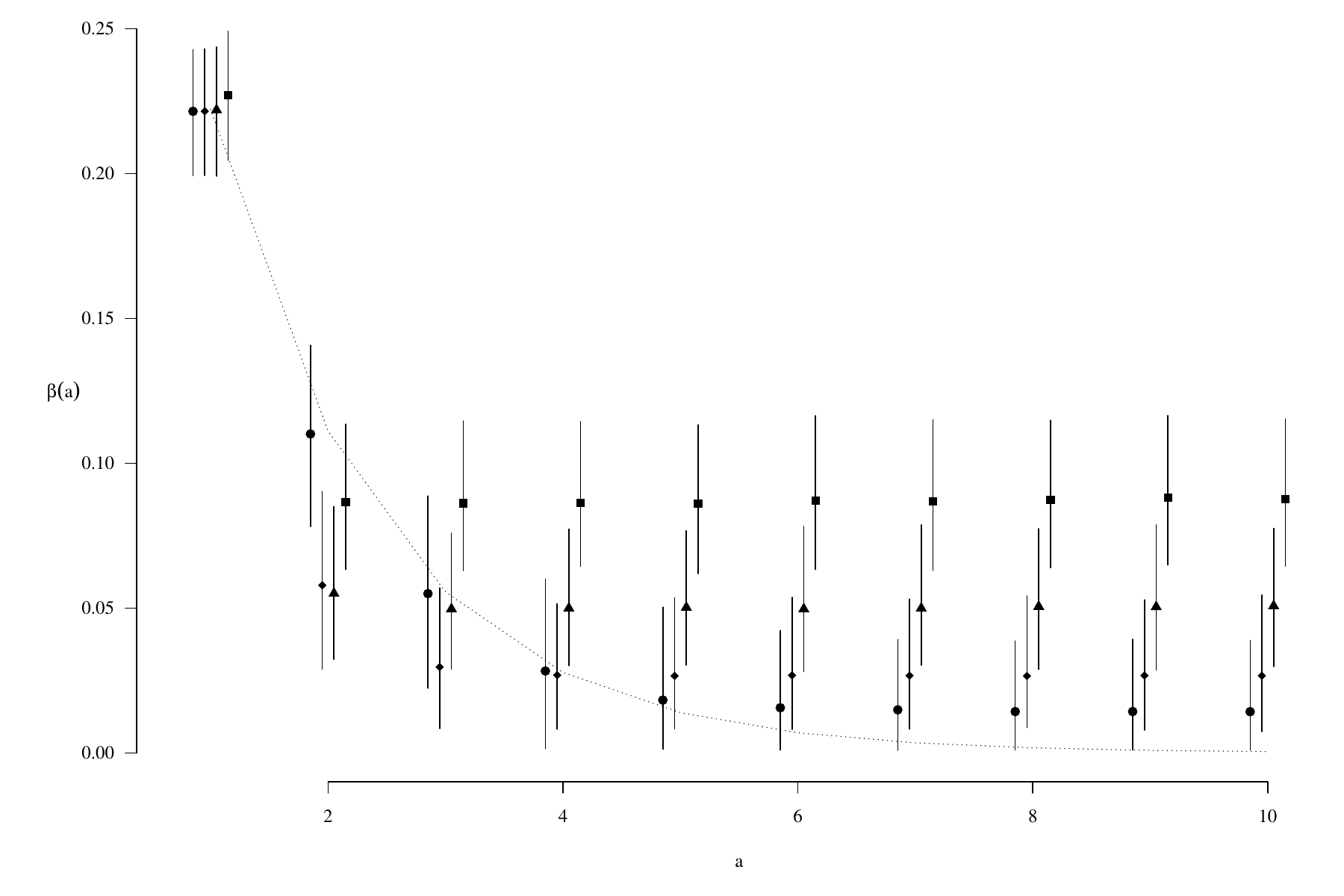}
  \caption{This figure illustrates the performance of our
    estimator for the two-state Markov chain depicted in
    \autoref{fig:markov}. We simulated length
    $n=1000$ chains and calculated $\widehat{\beta}^d(a)$ for $d=1$
    (circles), $d=2$ (diamonds), $d=3$ (triangles), and $d=4$ (squares). The dashed line
    indicates the true mixing coefficients. We show means and 95\%
    confidence intervals based on 1000 replications.}
  \label{fig:latent}
\end{figure}
\begin{figure}[t!]
  \centering
  \includegraphics[width=.9\textwidth]{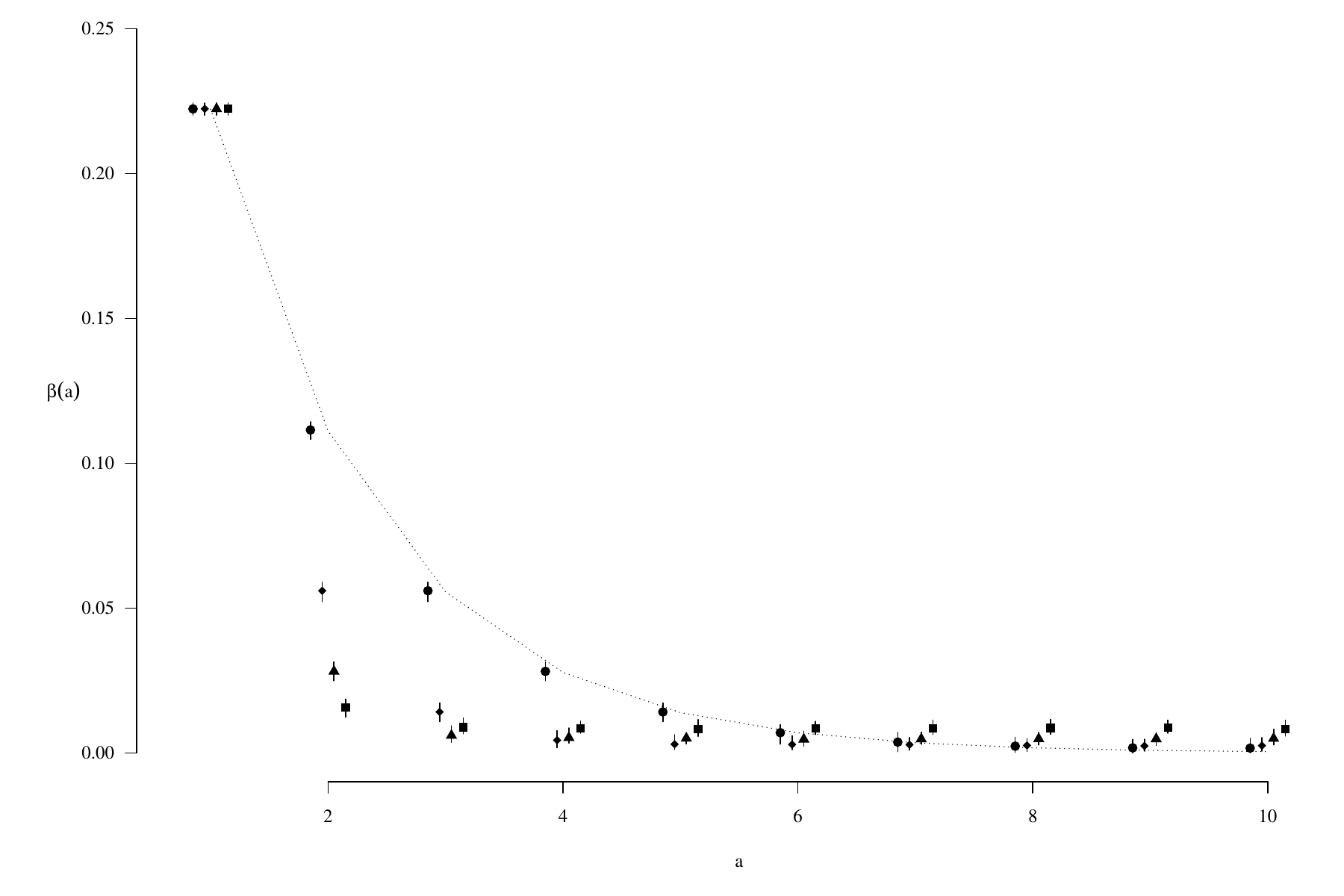}
  \caption{This again shows the two-state Markov chain but with length
    $n=100,000$ chains. Again, it shows $\widehat{\beta}^d(a)$ for $d=1$
    (circles), $d=2$ (diamonds), $d=3$ (triangles), and $d=4$ (squares). The dashed line
    indicates the true mixing coefficients. We show means and 95\%
    confidence intervals based on 100 replications.}
  \label{fig:latentbig}
\end{figure}

\subsubsection{Markov chain of order $m$}
Before examining a long-memory process, we simulate an intermediate
case. We construct a Markov model of order $m$ on $\{0,1\}$ using the
following transition probability:
\begin{align}
  \label{eq:10}
  P(Z_t = 1 | Z_{t-m},\ldots,Z_{t-1}) &= \frac{m-1}{m}
  (1-\xi_m) + \frac{1}{m}\xi_m &\mbox{with} &&\xi_m&=\frac{1}{m}\sum_{i=1}^m Z_{t-i}.
\end{align}
Essentially, this process avoids long strings of ones or zeros. In
this case, we have that $\beta(a) = \beta^m(a) = \beta^{m+k}(a)$ for
all $k \in \mathbb{N}$. Therefore, we should be able to estimate
$\beta(a)$ well by taking $d=m$. However, for smaller values of $d$,
we will tend to underestimate $\beta(a)$. In fact, it is possible,
using equation~\eqref{eq:2a}, to calculate $\beta^d(a)$ for each
$d=1,\ldots,m$. 
We simulated chains of length
$n=50000$ from this Markov chain with $m=4$. \autoref{fig:mMarkov} shows the
performance of the estimator based on 100 replications. Here, we allow
the Markov approximation to vary as 
$d\in\{1,\ 2,\ 3,\ 4,\ 5\}$, even though $d=m=4$ is exact. As above, the estimator performs
well for $a\leq 5$. Note that, for $d<m$, we can estimate $\beta^d(a)$
well as we would expect.
\begin{figure}[t!]
  \centering
  \includegraphics[width=.9\textwidth]{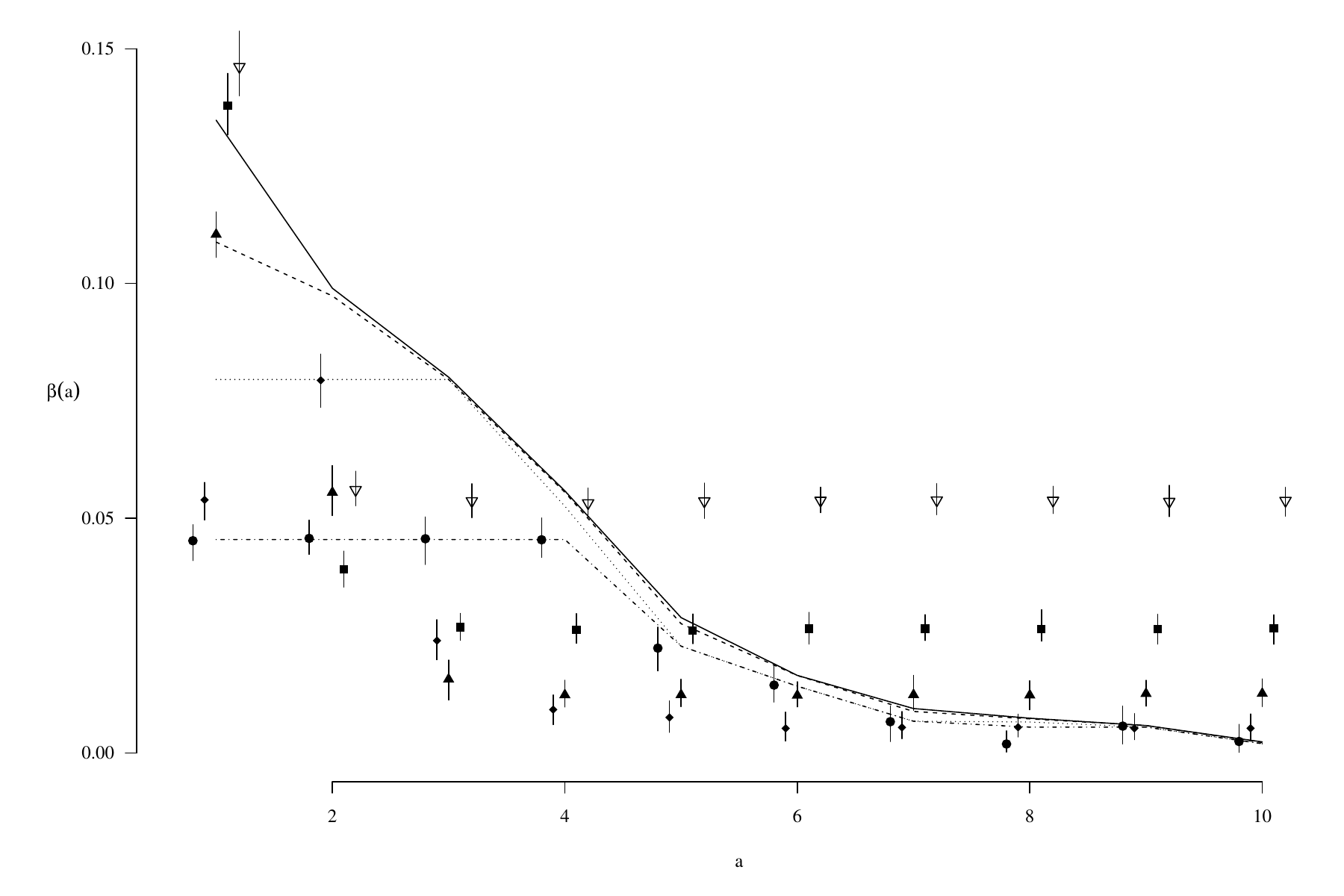}
  \caption{This figure illustrates the performance of our
    estimator for the two-state $m$-Markov chain generated with the
    transition probability in equation~\eqref{eq:10}. We simulated length
    $n=50000$ chains and calculated $\widehat{\beta}^d(a)$ for $d=1$
    (circles), $d=2$ (diamonds), $d=3$ (solid triangles), $d=4$
    (squares), and $d=5$ (open triangles). The solid line
    indicates $\beta(a)$. Other, lower-dimensional mixing coefficients
    are given by $\beta^1(a)$ (dot-dash), $\beta^2(a)$ (dotted), and
    $\beta^3(a)$ (dashed). We show means and 95\%
    confidence intervals based on 100 replications.}
  \label{fig:mMarkov}
\end{figure}

\subsubsection{Long-memory discrete process}
\label{sec:long-memory-discrete}

As an example of a long memory process, we construct,
following~\citet{Weiss1973}, a partially observable Markov process which we
call the ``even process''. Let $X_t$ be the observed sequence which takes as
input the Markov process $S_t$ constructed above. We observe
\begin{equation}
  \label{eq:7}
  X_t = \begin{cases} 1 & (S_t,S_{t-1}) = (A,B)\mbox{ or }(B,A)\\ 0
    & \mbox{else}. \end{cases}
\end{equation}
Since $S_t$ is Markovian, the joint process $(S_t,S_{t-1})$ is as well, so we
can calculate its mixing rate $\beta(a) =
\frac{8}{9}\left(\frac{1}{2}\right)^a$. The even process must also be
$\beta$-mixing, and at least as fast as the joint process, since it is a
measurable function of a mixing process. However, $X_t$ itself is
non-Markovian: runs of ones must have even lengths, so we need to know how many
ones have been observed to know whether the next observation can be zero or
must be a one. Thus, the true mixing coefficients are bounded above, though
unknown. Using the same procedure as above, \autoref{fig:even} shows the
estimated mixing coefficients. Again we observe a bias for $a$ large due to the
nonnegativity of the estimator.
\begin{figure}[t!]
  \centering
  \includegraphics[width=.9\textwidth]{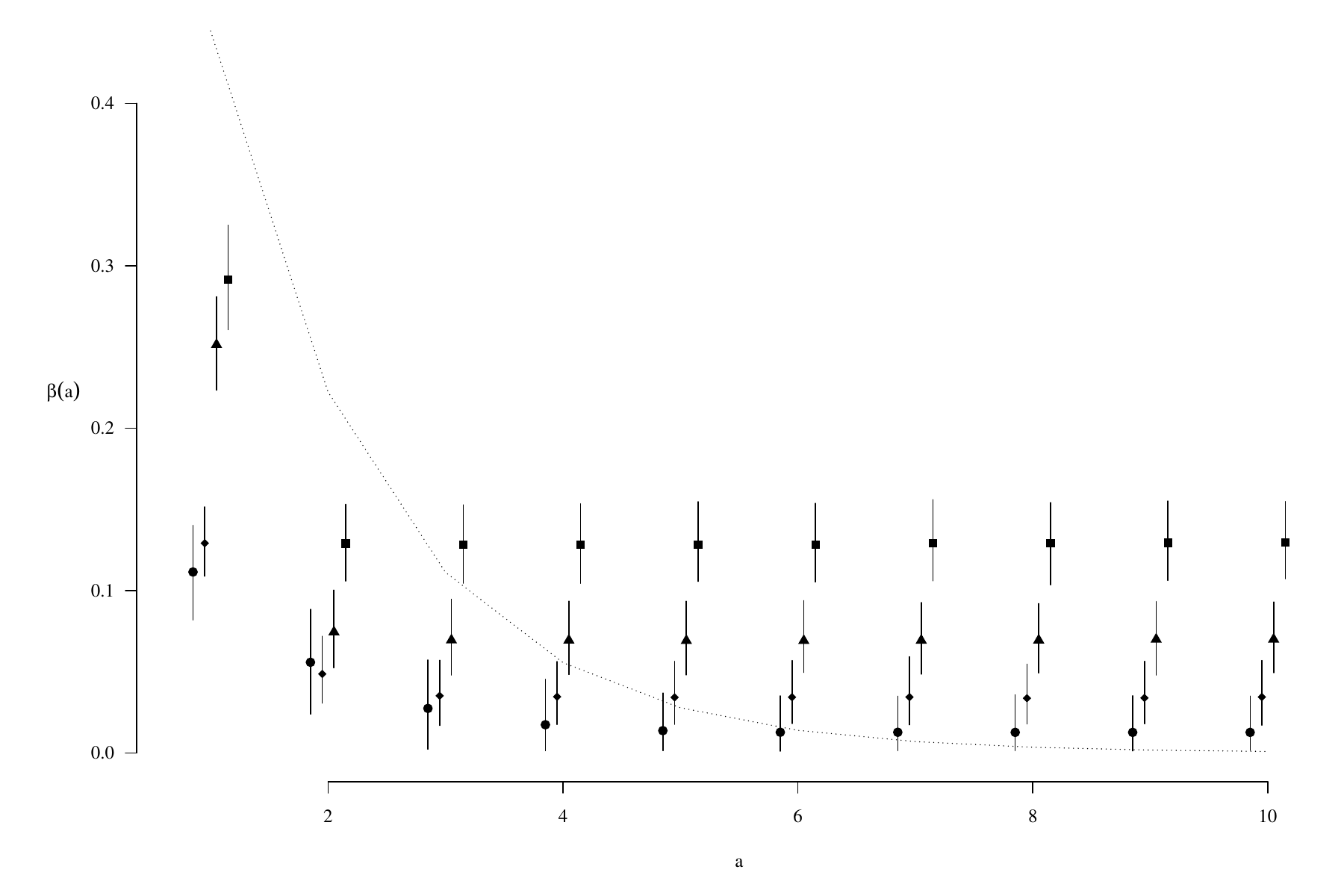}
  \caption{This figure illustrates the performance of our
    estimator for the even process in \autoref{eq:7}. Again, we simulated length
    $n=1000$ chains and calculated $\widehat{\beta}^d(a)$ for $d=1$
    (circles), $d=2$ (diamonds), $d=3$ (triangles), and $d=4$ (squares). The dashed line
    indicates an upper bound on the true mixing coefficients. We show
    means and 95\% confidence intervals based on 1000 replications.}
  \label{fig:even}
\end{figure}

\subsubsection{Autoregressive process}
\label{sec:autor-proc}

Finally, we estimate the $\beta$-mixing coefficients for an AR(1) model
\begin{align*}
  \label{eq:9}
  Z_t &= 0.5Z_{t-1} + \eta_t && \eta_t \overset{iid}{\sim}\mbox{N}(0,1).
\end{align*}
While, this process is Markovian, there is no closed form solution to
(\ref{eq:111}), so we calculate it via numerical
integration. \autoref{fig:ar} shows the performance of the
estimator for $d=1$. \autoref{fig:ar} shows the performance for
varying
$n\in\{10^2,\ 10^3,\ 10^4,\ 10^5,\ 10^6\}$. We select the bandwidth for
each $n$
using \autoref{alg:calibrate}. The selected numbers of bins are 2, 8,
17, 44, 90. As $n$ grows, the bias shrinks, even for large $a$ while
the variance of the estimators also shrinks rapidly. However, this
figure shows that even with large amounts of data, accurate estimation
is difficult.
\begin{figure}[t!]
  \centering
  \includegraphics[width=.9\textwidth]{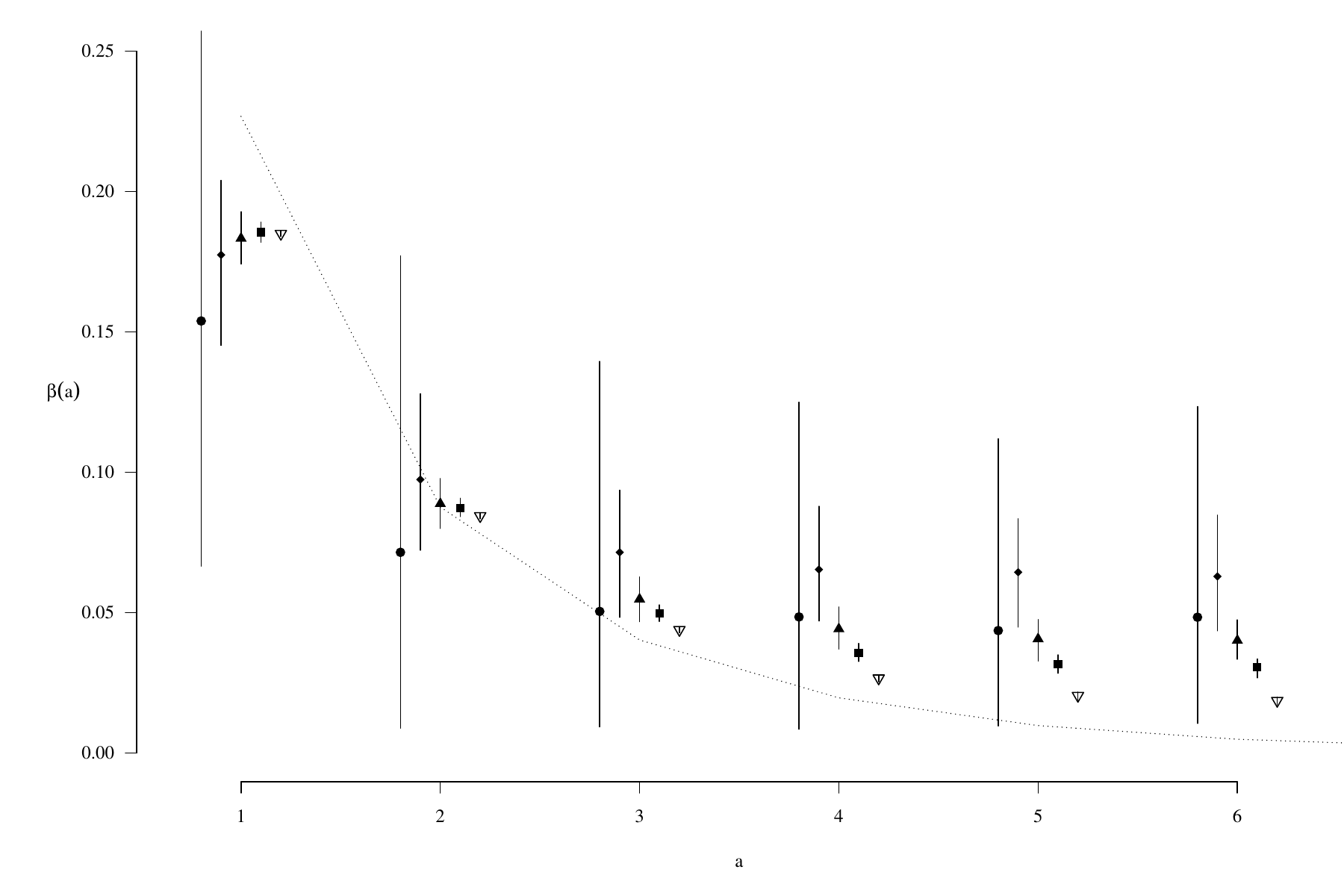}
  \caption{This figure illustrates the performance of our proposed
    estimator for the AR(1) model. We simulated chains of length
    $n\in\{10^2,\ 10^3,\ 10^4,\ 10^5,\ 10^6\}$ and calculated $\widehat{\beta}^1(a)$.
    The dashed line indicates the true mixing coefficients calculated
    via numerical integration. We show means and 95\% confidence
    intervals based on 250 replications.}
  \label{fig:ar}
\end{figure}

\subsection{Real data}
\label{sec:real-data}

To illustrate the performance of our estimator in applications, we
investigate an economic dataset in larger dimensions than in the
simulations above. We use a $q=6$-dimensional macroeconomic time series
which tracks recessions in various countries. In particular, we track
recession indicators in Canada, Germany, France, Great Britain, Japan,
and the United States. We chose this dataset for a number of reasons.
First, the data are publicly available from the \href{https://research.stlouisfed.org/fred2/}{Federal
Reserve Economic Database}
\begin{table}[t!]
  \centering
  \begin{tabular}{ll}
    \hline\hline
    Series ID & Country\\
    \hline
    CANRECDM & Canada\\
    DEURECDM & Germany\\
    FRARECDM & France\\
    GBRRECDM & Great Britain\\
    JPNRECDM & Japan\\
    USARECDM & United States\\
    \hline\hline
  \end{tabular}
  \caption{Economic recession data from \href{https://research.stlouisfed.org/fred2/}{FRED}}
  \label{tab:recessionNames}
\end{table}
using the series presented in \autoref{tab:recessionNames}.
Second, the series is long, providing daily
observations from December 1, 1961 until September of 2014 for a total
of $n=19288$ observations. This will enable us to allow $d$ to grow
quite quickly.  Third, the data are binary, so mixing
coefficients may be a more reasonable measure of temporal dependence
than, say, correlation. It also means we can ignore the issue of bin
selection.  Fourth, the data likely have high temporal
dependence as the indicators are based on a combination of monthly and
quarterly macroeconomic aggregates such as gross domestic product,
inflation, and unemployment. This means that using
both large $d$ and very large $a$ is necessary. Finally, the data are strongly
cross-sectionally dependent since these are all developed countries
likely to enter recession or expansion at similar times. This
cross-sectional dependence makes it unreasonable to examine each
series individually.

With six dimensions, the curse of dimensionality is immediately an
issue: $\widehat{f}^{2d}$ with $2$ bins along each dimension will have
$2^{12\gamma}$ bins when the Markov approximation is of length
$\gamma$ (that is $d=q\gamma=6\gamma$). In
\autoref{fig:recession}, we present estimated mixing coefficients for
$a$ between 1 and $360$ (giving estimates for 1-month, 2-month, up to
1-year lag dependence) and $\gamma \in\{1,\ 2,\ 5,\ 10,\ 20\}$. As the
figure illustrates, these data are highly temporally dependent. The
coefficients for $\gamma=1$ decrease smoothly in the lag $a$ while the
estimates for larger $\gamma$ behave less well. However, only the case of
$\gamma=20$ seems to exhibit the strong upward bias we might expect if $\gamma$
is large relative to $n$. Note that in this case we are estimating the
differences in probabilities in $2^{240}$ bins, although many of these
will be empty under both distributions.
\begin{figure}[t!]
  \centering
  \includegraphics[width=.9\textwidth]{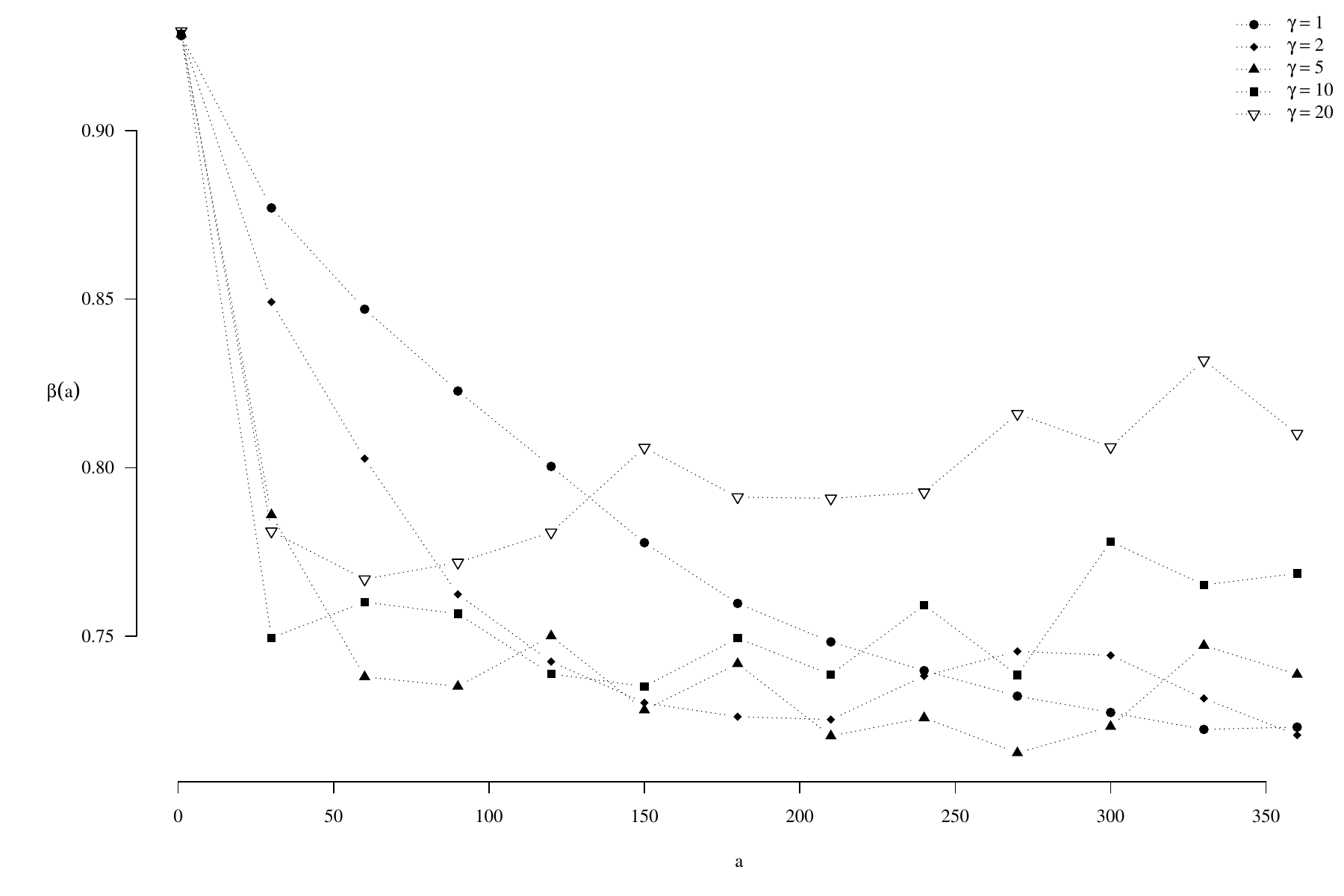}
  \caption{Estimated mixing coefficients for the recession
    indicators.}
  \label{fig:recession}
\end{figure}

\section{Discussion}
\label{sec:discussion}

We have shown that our estimator of the $\beta$-mixing coefficients is
consistent for the true coefficients $\beta(a)$ under some conditions on the
data-generating process. There are numerous results in the statistics
literature which assume knowledge of the
$\beta$-mixing coefficients, yet as far as we know, this is the first estimator for
them. An ability to estimate these coefficients will allow researchers
to apply existing results to dependent data without the need to arbitrarily
assume their values. Additionally, it will allow probabilists to
recover unknown mixing coefficients for stochastic processes via
simulation. Despite the obvious utility of this estimator, as
a consequence of its novelty, it comes with
a number of potential extensions which warrant careful exploration as well as
some drawbacks.

Several other mixing and weak-dependence coefficients also have a
total-variation flavor, perhaps most notably $\alpha$-mixing
\cite{Doukhan1994,DedeckerDoukhan2007,Bradley2005}.  None of them have
estimators, yet, and the same trick might well work for them, too.

The reader will note that \autoref{thm:two} does not provide a convergence
rate.  The rate in \autoref{thm:markov-rate} applies only to Markov processes
or the difference between $\widehat{\beta}^d(a)$ and $\beta^d(a)$.  In order to
provide a rate in \autoref{thm:two}, we would need a better understanding of
the non-stochastic convergence of $\beta^d(a)$ to $\beta(a)$. It is not
immediately clear that this quantity can converge at any well-defined rate. In
particular, it seems plausible, but is not proven, that the rate of convergence
depends on the tail of the sequence $\{\beta(a)\}_{a=1}^\infty$.

The use of histograms rather than kernel density estimators for the joint and
marginal densities is surprising and perhaps not ultimately necessary. As
mentioned above, \citet{Tran1989} proved that KDEs are consistent for
estimating the stationary density of a time series with $\beta$-mixing inputs,
so perhaps one could replace the histograms in our estimator with
KDEs. However, this would need an analogue of the double asymptotic results
proven for histograms in \autoref{lem:three}. In particular, we need to
estimate increasingly higher dimensional densities as
$n\rightarrow\infty$. This does not cause a problem of small-$n$-large-$d$
since $d$ is chosen as a function of $n$, however it will lead to increasingly
higher dimensional integration. For histograms, the integral is always
computationally trivial, but in the case of KDEs, the numerical accuracy of the
integration algorithm becomes increasingly hard to assure. This issue could
swamp any statistical efficiency gains obtained through the use of kernels,
though further investigation is warranted.

The main drawback of an estimator based on a density estimate is its
complexity. The mixing coefficients are functionals of the joint and marginal
distributions derived from the stochastic process $\mathbf{X}$, however, it is
unsatisfying to estimate densities and calculate integrals in order to estimate a
single number. Vapnik's main principle for solving problems using a restricted
amount of information is ``When solving a given problem, try to avoid solving a
more general problem as an intermediate step~\citep[p.~30]{Vapnik2000}.''
However, despite our estimator's complexity, we are able to obtain nearly
parametric rates of convergence to the Markov approximation departing only by
logarithmic factors. While the simplicity principle is clearly violated,
perhaps our seed will precipitate a more aesthetically pleasing solution.

\section*{Acknowledgements}
\label{sec:acknowledgements}

The authors are grateful to Darren Homrighausen for providing useful
insights. We also thank two anonymous reviewers for helpful comments on an
earlier version of this paper. Finally, we would like to thank the Institute
for New Economic Thinking for financial support.  In addition, DJM
acknowledges support from the NSF (DMS 1407439) and CRS
acknowledges support from grants of the NIH (R01 NS047493) and the NSF
(DMS 1207759, 1418124).

\bibliography{betamixFinal.bib}

\end{document}